\documentclass[reqno,english]{amsart}%
\usepackage{amsfonts,amsmath,latexsym,verbatim,amscd,mathrsfs,color,array}
\usepackage[colorlinks=true]{hyperref}
\usepackage{amsmath,amssymb,amsthm,amsfonts,graphicx,color}
\usepackage{amssymb}
\usepackage{pdfsync}
\usepackage{epstopdf}
\usepackage{cite}
\usepackage{graphicx}
\usepackage{amsmath}
\usepackage{amsfonts}%
\providecommand{\U}[1]{\protect\rule{.1in}{.1in}}
\numberwithin{equation}{section}

\newtheorem{theorem}{Theorem}[section]
\newtheorem{corollary}{Corollary}[section]
\newtheorem{lemma}{Lemma}[section]
\newtheorem{proposition}{Proposition}[section]
\newtheorem{remark}{Remark}[section]
\newtheorem{definition}{Definition}[section]

\numberwithin{equation}{section}

\newcommand{\bbr}{\mathbb{R}}

\newcommand{\bbn}{\mathbb{N}}
\newcommand{\ve}{\varepsilon}
\newcommand{\bd}{\begin{definition}}
\newcommand{\ed}{\end{definition}}

\newcommand{\br}{\begin{remark}}
\newcommand{\er}{\end{remark}}

\newcommand{\be}{\begin{equation}}
\newcommand{\ee}{\end{equation}}

\newcommand{\bc}{\begin{corollary}}
\newcommand{\ec}{\end{corollary}}
\begin{document}

\title[Caffarelli-Kohn-Nirenberg inequality]{Stability of the Caffarelli-Kohn-Nirenberg inequality: the existence of minimizers}

\author[J. Wei]{Juncheng Wei}
\address{\noindent Department of Mathematics, University of British Columbia,
Vancouver, B.C., Canada, V6T 1Z2}
\email{jcwei@math.ubc.ca}

\author[Y.Wu]{Yuanze Wu}
\address{\noindent  School of Mathematics, China
University of Mining and Technology, Xuzhou, 221116, P.R. China }
\email{wuyz850306@cumt.edu.cn}

\begin{abstract}
In this paper, we consider the following variational problem:
\begin{eqnarray*}
\inf_{u\in D^{1,2}_a(\bbr^N)\backslash\mathcal{Z}}\frac{\|u\|^2_{D^{1,2}_a(\bbr^N)}-C_{a,b,N}^{-1}\|u\|^2_{L^{p+1}(|x|^{-b(p+1)},\bbr^N)}}{dist_{D^{1,2}_{a}}^2(u, \mathcal{Z})}:=c_{BE},
\end{eqnarray*}
where $N\geq2$, $b_{FS}(a)<b<a+1$ for $a<0$ and $a\leq b<a+1$ for $0\leq a<a_c:=\frac{N-2}{2}$ and $a+b>0$ with $b_{FS}(a)$ being the Felli-Schneider curve, $p=\frac{N+2(1+a-b)}{N-2(1+a-b)}$, $\mathcal{Z}= \{ c \tau^{a_c-a}W(\tau x)\mid c\in\bbr\backslash\{0\}, \tau>0\}$ and up to dilations and scalar multiplications, $W(x)$, which is positive and radially symmetric, is the unique extremal function of the following classical Caffarelli-Kohn-Nirenberg (CKN for short) inequality
\begin{eqnarray*}
\bigg(\int_{\bbr^N}|x|^{-b(p+1)}|u|^{p+1}dx\bigg)^{\frac{2}{p+1}}\leq C_{a,b,N}\int_{\bbr^N}|x|^{-2a}|\nabla u|^2dx
\end{eqnarray*}
with $C_{a,b,N}$ being the optimal constant.  It is known in \cite{WW2022} that $c_{BE}>0$.  In this paper, we prove that the above variational problem has a minimizer for $N\geq2$ under the following two assumptions:
\begin{enumerate}
\item[$(i)$]\quad $a_c^*\leq a<a_c$ and $a\leq b<a+1$,
\item[$(ii)$]\quad $a<a_c^*$ and $b_{FS}^*(a)\leq b<a+1$,
\end{enumerate}
where $a_c^*=\bigg(1-\sqrt{\frac{N-1}{2N}}\bigg)a_c$ and 
\begin{eqnarray*}
b_{FS}^*(a)=\frac{(a_c-a)N}{a_c-a+\sqrt{(a_c-a)^2+N-1}}+a-a_c.
\end{eqnarray*}
Our results extend that of Konig in \cite{K2023} for the Sobolev inequality to the CKN inequality.  Moreover, we believe that our assumptions~$(i)$ and $(ii)$ are optimal for the existence of minimizers of the above variational problem.

\vspace{3mm} \noindent{\bf Keywords:} Caffarelli-Kohn-Nirenberg inequality; Bianchi-Egnell type stability; Existence of minimizers.

\vspace{3mm}\noindent {\bf AMS} Subject Classification 2010: 35B09; 35B33; 35B40; 35J20.%

\end{abstract}

\date{}
\maketitle

\section{Introduction}
Let $D^{1,2}_{a}(\bbr^N)$ be the Hilbert space given by
\begin{eqnarray}\label{eqn886}
D^{1,2}_{a}(\bbr^N)=\{u\in D^{1,2}(\bbr^N)\mid \int_{\bbr^N}|x|^{-2a}|\nabla u|^2dx<+\infty\}
\end{eqnarray}
with the inner product
\begin{eqnarray*}
\langle u,v \rangle_{D^{1,2}_{a}(\bbr^N)}=\int_{\bbr^N}|x|^{-2a}\nabla u\nabla vdx
\end{eqnarray*}
and $D^{1,2}(\bbr^N)=\dot{W}^{1,2}(\bbr^N)$ being the usual homogeneous Sobolev space (cf. \cite[Definition~2.1]{FG2021}).   Then the classical Caffarelli-Kohn-Nirenberg (CKN for short) inequality, established by Caffarelli, Kohn and Nirenberg in the celebrated paper \cite{CKN1984} in a more general version, states that
\begin{eqnarray}\label{eq0001}
\bigg(\int_{\bbr^N}|x|^{-b(p+1)}|u|^{p+1}dx\bigg)^{\frac{2}{p+1}}\leq C_{a,b,N}\int_{\bbr^N}|x|^{-2a}|\nabla u|^2dx,
\end{eqnarray}
for all $u\in D_a^{1,2}(\bbr^N)$, where $-\infty<a<a_c:=\frac{N-2}{2}$, $p=\frac{N+2(1+a-b)}{N-2(1+a-b)}$ and 
\begin{eqnarray*}
\left\{\aligned&a+\frac{1}{2}<b<a+1,\quad N=1,\\
&a<b<a+1,\quad N=2,\\
&a\leq b<a+1,\quad N\geq3.
\endaligned
\right.
\end{eqnarray*}
Here, for the sake of simplicity, we denote $a_c=\frac{N-2}{2}$, as that in \cite{DELT2009,DEL2012,DEL2016}.

\vskip0.12in

As that of many famous functional inequalities such as the Sobolev inequality, the Hardy-Littlewood-Sobolev inequality, the Gagliardo-Nirenberg-Sobolev inequality, the Euclidean logarithmic Sobolev inequality and so on, as a generalization of the Sobolev and Hardy-Sobolev inequalities, the CKN inequality~\eqref{eq0001} is also very helpful in understanding various problems in lots of mathematical fields, such as nonlinear partial differential equations, calculus of variations, geometric analysis, the theory of probability to mathematical physics and so on.  For this purpose, a fundamental task in understanding the CKN inequality~\eqref{eq0001} is to study the optimal constant, the classification of extremal functions, as well as their qualitative properties for parameters in the full region.  Under the above conditions, it is well known (cf. \cite{A1976,CC1993,CW2001,L1983,T1976}) that the CKN inequality~\eqref{eq0001} has extremal functions if and only if under the following assumptions:
\begin{enumerate}
\item\quad $a<b<a+1$ and $a<0$ for $N\geq2$,
\item\quad $a+\frac12<b<a+1$ and $a<0$ for $N=1$,
\item\quad $a\leq b<a+1$ and $0\leq a<a_c$ for $N\geq3$.
\end{enumerate}
Moreover, let
\begin{eqnarray}\label{eqn991}
b_{FS}(a)=\frac{N(a_c-a)}{2\sqrt{(a_c-a)^2+(N-1)}}+a-a_c>a
\end{eqnarray}
be the Felli-Schneider curve found in \cite{FS2003},
then it is also well known (cf. \cite{A1976,CC1993,DELT2009,DEL2012,DEL2016,DET2008,FS2003,L1983,T1976}) that up to dilations $\tau^{a_c-a}u(\tau x)$ and scalar multiplications $Cu(x)$ (also up to translations $u(x+y)$ in the special case $a=b=0$), the CKN inequality~\eqref{eq0001} has a unique extremal function
\begin{eqnarray}\label{eqq0097}
W(x)=(2(p+1)(a_c-a)^2)^{\frac{1}{(p-1)}}\bigg(1+|x|^{(a_c-a)(p-1)}\bigg)^{-\frac{2}{p-1}}
\end{eqnarray}
either for $b_{FS}(a)\leq b<a+1$ with $a<0$ or for $a\leq b<a+1$ with $0\leq a<a_c$ in the cases of $N\geq2$ while, extremal functions of \eqref{eq0001} must be non-radial either for the full region of $a$ and $b$ in the case of $N=1$ or for $a<b<b_{FS}(a)$ with $a<0$ in the cases of $N\geq2$.  Moreover, it has been proved in \cite{CW2001,LW2004} that there are exactly two extremal functions of \eqref{eq0001} in the case of $N=1$ up to dilations and scalar multiplications while in the cases of $N\geq2$, extremal functions of \eqref{eq0001} must have $\mathcal{O}(N-1)$ symmetry for $a<b<b_{FS}(a)$ with $a<0$, that is, extremal functions of \eqref{eq0001} for $N\geq2$ must depend on the radius $r$ and the angle $\theta_{N}$ between the positive $x_N$-axis and $\overrightarrow{Ox}$ for $a<b<b_{FS}(a)$ with $a<0$ up to rotations.  To our best knowledge, whether the extremal function of \eqref{eq0001} is unique or not for $a<b<b_{FS}(a)$ with $a<0$ in the cases of $N\geq2$ remains open.

\vskip0.12in

Besides the existence and classification of extremal functions and the computation of the optimal constant, a more interesting and challenging problem in understanding functional inequalities is its quantitative stability, whose basic question one wants to address in this aspect is the following (cf. \cite{F2013}):
\begin{enumerate}
\item[$(Q)$]\quad Suppose we are given a functional inequality for which minimizers are known.  Can we prove, in some quantitative way, that if a function ``almost attains the
equality'' then it is close (in some suitable sense) to one of the minimizers?
\end{enumerate}
The studies on the quantitative stability of functional inequalities were initialed by Brezis and Lieb in \cite{BL1985} by raising an open question for the classical Sobolev inequality ($a=b=0$ in the CKN inequality~\eqref{eq0001}),
\begin{eqnarray}\label{eqq0093}
S_N\bigg(\int_{\bbr^N}|u|^{\frac{2N}{N-2}}dx\bigg)^{\frac{N-2}{N}}\leq \int_{\bbr^N}|\nabla u|^2dx
\end{eqnarray}
for $N\geq3$, which was settled by Bianchi and Egnell in \cite{BE1991} by proving that
\begin{eqnarray}\label{eqq0090}
dist_{D^{1,2}}^2(u, \mathcal{U})\lesssim\|u\|^2_{D^{1,2}(\bbr^N)}-S_{N}\|u\|^2_{L^{\frac{2N}{N-2}}(\bbr^N)},
\end{eqnarray}
where $S_N$ is the optimal constant of \eqref{eqq0093} and
\begin{eqnarray*}
\mathcal{U}=\{cU_{y,\lambda}\mid c\in\bbr\backslash\{0\}, \lambda>0\text{ and }y\in\bbr^N\}
\end{eqnarray*}
with $U(x)$ being the Aubin-Talanti bubble (cf. \cite{A1976,T1976}) and $U_{y,\lambda}(x)=\lambda^{\frac{N-2}{2}}U(\lambda(x-y))$.  Since then, the stability of functional inequalities, which is similar to \eqref{eqq0090}, is called the Bianchi-Egnell type stability.  In the very recent paper \cite{WW2022}, we prove the following Bianchi-Egnell type stability of the CKN inequality~\eqref{eq0001}:
\begin{theorem}\label{thm0001}
Let $N\geq3$, $a<a_c$ and $b_{FS}(a)$
be the Felli-Schneider curve given by \eqref{eqn991} and assume that either
\begin{enumerate}
\item[$(1)$]\quad $b_{FS}(a)< b<a+1$ with $a<0$ or
\item[$(2)$]\quad $a\leq b<a+1$ with $a\geq0$ and $a+b>0$.
\end{enumerate}
Then
\begin{eqnarray}\label{eqq0091}
dist_{D^{1,2}_{a}}^2(u, \mathcal{Z})\lesssim\|u\|^2_{D^{1,2}_a(\bbr^N)}-C_{a,b,N}^{-1}\|u\|^2_{L^{p+1}(|x|^{-b(p+1)},\bbr^N)}
\end{eqnarray}
for all $u\in D^{1,2}_a(\bbr^N)$, where 
\begin{eqnarray*}
\mathcal{Z}=\{cW_\tau(x)\mid c\in\bbr\backslash\{0\}\text{ and }\tau>0\}
\end{eqnarray*}
with $W_\tau(x)=\tau^{a_c-a}W(\tau x)$,
$L^{p+1}(|x|^{-b(p+1)},\bbr^N)$ is the usual weighted Lebesgue space with its usual norm given by
\begin{eqnarray*}
\|u\|_{L^{p+1}(|x|^{-b(p+1)},\bbr^N)}=\bigg(\int_{\bbr^N}|x|^{-b(p+1)}|u|^{p+1}dx\bigg)^{\frac{1}{p+1}}.
\end{eqnarray*}
\end{theorem}
Moreover, by the same argument as that used in \cite{WW2022} for proving Theorem~\ref{thm0001}, it is not difficult to obtain the following one.
\begin{theorem}\label{thm0002}
Let $N=2$ and assume that $b_{FS}(a)< b<a+1$ with $a<0$ and $b_{FS}(a)$
being the Felli-Schneider curve given by \eqref{eqn991}.
Then
\begin{eqnarray*}
dist_{D^{1,2}_{a}}^2(u, \mathcal{Z})\lesssim\|u\|^2_{D^{1,2}_a(\bbr^2)}-C_{a,b,2}^{-1}\|u\|^2_{L^{p+1}(|x|^{-b(p+1)},\bbr^2)}
\end{eqnarray*}
for all $u\in D^{1,2}_a(\bbr^2)$.
\end{theorem}

\begin{remark}
In \cite[Theorem~1.1]{WW2022}, we have claimed that Theorem~\ref{thm0001} holds true for $b_{FS}(a)\leq b<a+1$ with $a<0$.  However, it is incorrect for $b=b_{FS}(a)$ since it has been proved in \cite{FS2003} that $W$ is degenerate for $b=b_{FS}(a)$. It follows that the spectral gap inequality,
\begin{eqnarray*}
\|\rho\|_{D^{1,2}_a(\bbr^N)}^2\geq(p+\ve)\int_{\bbr^N}|x|^{-b(p+1)}W^{p-1}\rho^2dx
\end{eqnarray*}
for all $\rho\in\mathcal{N}^{\perp}$ which plays the key role in proving the Bianchi-Egnell type stability of the CKN inequality in \cite[Theorem~1.1]{WW2022}, does not hold true any more for $b=b_{FS}(a)$, where $\ve>0$ is a fixed small constant and
\begin{eqnarray*}
\mathcal{N}^{\perp}=\{\rho\in D^{1,2}_a(\bbr^N)\mid \langle \rho,W\rangle_{D^{1,2}_a(\bbr^N)}=\langle \rho,W' \rangle_{D^{1,2}_a(\bbr^N)}=0\}.
\end{eqnarray*}
This has already been observed in \cite{DT2023,FP2023}.  We now correct \cite[Theorem~1.1]{WW2022} here to Theorem~\ref{thm0001}.  Moreover, it is worth pointing out that in this degenerate situation, the Bianchi-Egnell stability of the CKN inequality still holds true for higher order of the distance functional, as proved in \cite{FP2023} whose ideas can be traced back to \cite{F2022}.
\end{remark}

\vskip0.12in

The power two of the distance in the left side of the Bianchi-Engell inequality~\eqref{eqq0090} is well known to be optimal, which is also the case of the Bianchi-Egnell type stability of the CKN inequality~\eqref{eq0001} given by \eqref{eqq0091}.  Thus, we can define the following two variational problems:
\begin{eqnarray}\label{eqq0092}
\inf_{u\in D^{1,2}(\bbr^N)\backslash\mathcal{U}}\frac{\|u\|^2_{D^{1,2}(\bbr^N)}-S_{N}\|u\|^2_{L^{\frac{2N}{N-2}}(\bbr^N)}}{dist_{D^{1,2}}^2(u, \mathcal{U})}:=s_{BE}
\end{eqnarray}
and
\begin{eqnarray}\label{eqq1001}
\inf_{u\in D^{1,2}_a(\bbr^N)\backslash\mathcal{Z}}\frac{\|u\|^2_{D^{1,2}_a(\bbr^N)}-C_{a,b,N}^{-1}\|u\|^2_{L^{p+1}(|x|^{-b(p+1)},\bbr^N)}}{dist_{D^{1,2}_{a}}^2(u, \mathcal{Z})}:=c_{BE},
\end{eqnarray}
By \eqref{eqq0090}, we know that $s_{BE}>0$, and by Theorems~\ref{thm0001} and \ref{thm0002}, we know that $c_{BE}>0$ either for
\begin{enumerate}
\item[$(1)$]\quad $b_{FS}(a)< b<a+1$ with $a<0$ or for
\item[$(2)$]\quad $a\leq b<a+1$ with $a\geq0$ and $a+b>0$.
\end{enumerate}
in the cases of $N\geq2$.  Moreover, as pointed out by Konig in \cite{K2023}, it is a long-standing open question that is to determine the best constant $s_{BE}$ in \eqref{eqq0092}.  It has been proved by Konig in \cite{K2023} that the variational problem~\eqref{eqq0092} has a minimizer which makes the key step in determining the best constant $s_{BE}$ and gives a positive answer to the open question proposed in \cite{DEFFL2022}.  Konig's proof in \cite{K2023} is a concentration-compactness type argument on the distance functional $dist_{D^{1,2}}^2(u, \mathcal{U})$.  By establishing two crucial energy estimates of $s_{BE}$ (cf. \cite{K2022,K2023}), Konig excluded the dichotomy case and the vanishing case of $dist_{D^{1,2}}^2(u, \mathcal{U})$ and proved that the variational problem~\eqref{eqq0092} has a minimizer.    In this argument, a key integrant is the well understanding of the spectrum of the Laplacian operator $-\Delta$ in the weighted Lebesgue space $L^2(U^{\frac{4}{N-2}},\bbr^N)$, which is crucial in establishing one of the two important energy estimates of $s_{BE}$.  Furthermore, as pointed out by Konig in \cite{K2023}, this strategy also works for the fractional Sobolev inequality of $N\geq2$.

\vskip0.12in

In this paper, by making a well understanding of the spectrum of the operator $-div(|x|^{-a}\nabla\cdot)$ in the weighted Lebesgue space $L^2(|x|^{-b(p+1)}W^{p-1},\bbr^N)$, we adapt Konig's strategy in \cite{K2023} to prove the following theorem.
\begin{theorem}\label{thmq0001}
Let $N\geq2$ and assume that either
\begin{enumerate}
\item[$(1)$]\quad $b_{FS}(a)< b<a+1$ with $a<0$ or
\item[$(2)$]\quad $a\leq b<a+1$ with $0\leq a<a_c$ and $a+b>0$.
\end{enumerate}
Then the variational problem~\eqref{eqq1001} has a minimizer, provided
\begin{enumerate}
\item[$(i)$]\quad $a_c^*<a<a_c$ and $a\leq b<a+1$,
\item[$(ii)$]\quad $a\leq a_c^*$ and $b_{FS}^*(a)\leq b<a+1$,
\end{enumerate}
where 
\begin{eqnarray*}
a_c^*=\bigg(1-\sqrt{\frac{N-1}{2N}}\bigg)a_c\quad\text{and}\quad
b_{FS}^*(a)=\frac{(a_c-a)N}{a_c-a+\sqrt{(a_c-a)^2+N-1}}+a-a_c.
\end{eqnarray*}
\end{theorem}

\vskip0.12in

Theorem~\ref{thmq0001} is the generalization of Konig's reuslt in \cite{K2023} for the Sobolev inequality~\eqref{eqq0093} to the CKN inequality~\eqref{eq0001}.  However, Konig has proved in \cite{K2023} that $s_{BE}$ is attained for all $N\geq3$ (even in the fractional setting for all $N\geq2$), while in Theorem~\ref{thmq0001}, we only prove that $c_{BE}$ is attained for $N\geq2$ under the assumptions~$(i)$ and $(ii)$ which do not cover the full region of the parameters $a,b$ under the conditions~$(1)$ and $(2)$.  The main reason is that under the assumptions~$(i)$ and $(ii)$ for $N\geq2$, the spectral gap inequality of the operator $-div(|x|^{-a}\nabla\cdot)$ in the weighted Lebesgue space $L^2(|x|^{-b(p+1)}W^{p-1},\bbr^N)$ is attained by a unique (up to scalar multiplications) function which is related to spherical harmonics on $\mathbb{S}^{N-1}$ of degree $0$, while in the remaining case, that is, $N\geq2$ with $a<a_c^*$ and $b_{FS}(a)<b<b_{FS}^*(a)$,
the spectral gap inequality of the operator $-div(|x|^{-a}\nabla\cdot)$ in the weighted Lebesgue space $L^2(|x|^{-b(p+1)}W^{p-1},\bbr^N)$ is attained by the functions which are related to spherical harmonics on $\mathbb{S}^{N-1}$ of degree $1$.  Thus, Konig's strategy in \cite{K2023}, that is, expansing the related functional at the possible best choice of test functions up to the third order term to derive a crucial energy estimate, works for $c_{BE}$ under the assumptions~$(i)$ and $(ii)$ since the expansion has a negative third order term (see the proof of Proposition~\ref{propq0002}) and is invalid in the remaining case since the expansion has a varnishing third order term (see the appendix).
To go further, we expand the functional of \eqref{eqq1001} at the possible best choice of test functions up to the fourth order term.  After tedious computations, we find that the possible best choice of test functions can not derive the desired energy estimate of $c_{BE}$ in this situation any more, since this expansion has a varnishing third order term and a positive fourth order term (see the appendix for more details). Taking into account the fact that the test functions are possible to be optimal, we believe that $c_{BE}$ will be not attained in this remaining case.  We remark that a similar situation is also faced in proving the existence of minimizers of $s_{BE}$ in the fractional setting for $N=1$, see, for example, the very recent paper \cite{K2023-1}.

\vskip0.12in

In the final of the introduction, we would like to point out that the studies on the stability of functional inequalities are growingly interested in recent years in the community of nonlinear analysis by its deep connections to many nonlinear partial differential equations, such as the fast diffusion equation, the Keller-Segel equation and so on.  We refer the readers to the survey \cite{DE2022} and the Lecture notes \cite{F2023} for their detailed introductions and references about the studies on lots of famous functional inequalities and their stability, such as the Sobolev inequality, the Hardy-Littlewood-Sobolev inequality, the Gagliardo-Nirenberg-Sobolev inequality, the Caffarelli-Kohn-Nirenberg inequality, the Euclidean logarithmic Sobolev inequality and so on.  We also would like to refer the readers to the note \cite{F2013} for the related studies on the stability of many geometric inequalities.

\vskip0.12in

\noindent{\bf\large Notations.} Throughout this paper, $a\sim b$ means that $C'b\leq a\leq Cb$ and $a\lesssim b$ means that $a\leq Cb$.  Moreover, $\bbn_0=\{0\}\cup\bbn$.

\section{Preliminaries}
Let $D^{1,2}_{a}(\bbr^N)$ be the Hilbert space given by \eqref{eqn886} with the norm $\|\cdot\|_{D^{1,2}_a(\bbr^N)}$, then by \cite[Proposition~2.2]{CW2001}, $D^{1,2}_{a}(\bbr^N)$ is isomorphic to the Hilbert space $H^1(\mathcal{C})$ through the Emden-Fowler transformation
\begin{eqnarray}\label{eq0007}
u(x)=|x|^{-(a_c-a)}v(-\ln|x|,\frac{x}{|x|}),
\end{eqnarray}
where $\mathcal{C}=\bbr\times \mathbb{S}^{N-1}$ is the standard cylinder, the inner product in $H^1(\mathcal{C})$ is given by
\begin{eqnarray*}
\langle w,v \rangle_{H^1(\mathcal{C})}&=&\int_{\mathcal{C}}\partial_t w\partial_t v+\nabla_{\mathbb{S}^{N-1}}w\nabla_{\mathbb{S}^{N-1}}v+(a_c-a)^2 wv d\mu\\
&:=&\int_{\mathcal{C}}\nabla w\nabla v+(a_c-a)^2 wv d\mu
\end{eqnarray*}
with $d\mu$ being the volume element on $\mathcal{C}$ and $w,v\in H^1(\mathcal{C})$.

\vskip0.12in

The Euler-Lagrange equation of the CKN inequality~\eqref{eq0001} is given by
\begin{eqnarray}\label{eq0018}
-div(|x|^{-a}\nabla u)=|x|^{-b(p+1)}|u|^{p-1}u, \quad\text{in }\bbr^N.
\end{eqnarray}
It has been proved in \cite{CC1993,DEL2016} that $W(x)$, given by \eqref{eqq0097}, is the unique nonnegative solution of \eqref{eq0018} in $D^{1,2}_a(\bbr^N)$ either for $b_{FS}(a)\leq b<a+1$ with $a<0$ or for $a\leq b<a+1$ with $a\geq0$ and $a+b>0$ up to dilations $W_\tau=\tau^{a_c-a}W(\tau x)$ in the cases of $N\geq2$.  By the transformation~\eqref{eq0007}, $W(x)$ and \eqref{eq0018} are transformed into
\begin{eqnarray}\label{eq0026}
\Psi(t)=\bigg(\frac{(p+1)(a_c-a)^2}{2}\bigg)^{\frac{1}{p-1}}\bigg(cosh(\frac{(a_c-a)(p-1)}{2}t)\bigg)^{-\frac{2}{p-1}}
\end{eqnarray}
and
\begin{eqnarray}\label{eq0006}
-\Delta_{\mathbb{S}^{N-1}}v-\partial_t^2v+(a_c-a)^2v=|v|^{p-1}v, \quad\text{in }\mathcal{C},
\end{eqnarray}
respectively,
where $t=-\ln|x|$ and $\theta=\frac{x}{|x|}$ for $x\in\bbr^N\backslash\{0\}$ and $\Delta_{\mathbb{S}^{N-1}}$ is the Laplace-Beltrami operator on $\mathbb{S}^{N-1}$.  Moreover,
the CKN inequality~\eqref{eq0001} and the variational problem~\eqref{eqq1001} is transformed into
\begin{eqnarray}\label{eq0009}
C_{a,b,N}^{-1}=\inf_{v\in H^1(\mathcal{C})\backslash\{0\}}\frac{\|v\|^2_{H^{1}(\mathcal{C})}}{\|v\|^2_{L^{p+1}(\mathcal{C})}},
\end{eqnarray}
and
\begin{eqnarray}\label{eqq0001}
\inf_{v\in H^1(\mathcal{C})\backslash\mathcal{Y}}\frac{\|v\|^2_{H^1(\mathcal{C})}-C_{a,b,N}^{-1}\|v\|^2_{L^{p+1}(\mathcal{C})}}{dist_{H^1(\mathcal{C})}^2(v, \mathcal{Y})}:=c_{BE},
\end{eqnarray}
respectively,
where $L^{p+1}(\mathcal{C})$ is the usual Lebesgue space with its usual norm given by $\|u\|_{L^{p+1}(\mathcal{C})}=\bigg(\int_{\mathcal{C}}|u|^{p+1}d\mu\bigg)^{\frac{1}{p+1}}$ and
\begin{eqnarray}\label{eqq1030}
\mathcal{Y}=\{c\Psi_s(t)\mid c\in\bbr\backslash\{0\}\text{ and } s\in\bbr\}
\end{eqnarray}
with $\Psi_s(t)=\Psi(t-s)$.

\vskip0.12in

On the other hand, it has been proved in \cite{FS2003} that $W(x)$ is nondegenerate in $D^{1,2}_a(\bbr^N)$ either for $b_{FS}(a)<b<a+1$ with $a<0$ or for $a\leq b<a+1$ with $a\geq0$ and $a+b>0$ in the cases of $N\geq2$.  That is, up to scalar multiplications,
\begin{eqnarray}\label{eq0010}
V(x):=\nabla W(x)\cdot x+(a_c-a)W(x)=\frac{\partial}{\partial\lambda}(\lambda^{(a_c-a)}W(\lambda x))|_{\lambda=1}
\end{eqnarray}
is the only nonzero solution of the linearization of \eqref{eq0018} around $W$ in $D^{1,2}_a(\bbr^N)$ which is given by
\begin{eqnarray}\label{eq0017}
-div(|x|^{-a}\nabla u)=p|x|^{-b(p+1)}W^{p-1}u, \quad\text{in }\bbr^N.
\end{eqnarray}
By the transformation~\eqref{eq0007}, the linear equation~\eqref{eq0017} can be rewritten as follows:
\begin{eqnarray}\label{eq0016}
-\Delta_{\mathbb{S}^{N-1}}v-\partial_t^2v+(a_c-a)^2v=p\Psi^{p-1}v, \quad\text{in }\mathcal{C}.
\end{eqnarray}
By applying the transformation~\eqref{eq0007} on \eqref{eq0010}, we know that
\begin{eqnarray*}
\Psi_s'(t)=\Psi'(t-s)=\frac{\partial}{\partial t}\Psi(t-s)=-\frac{\partial}{\partial s}\Psi(t-s)
\end{eqnarray*}
is the only nonzero solution of \eqref{eq0016} in $H^1(\mathcal{C})$.

\vskip0.2in

\section{Spectral gap inequality}
We denote by $\mathcal{M}=\bbr \Psi\bigoplus\bbr\Psi'$.  Since $\Psi$ is Morse index $1$, by the nondegeneracy of $\Psi$ under the conditions~$(1)$ and $(2)$ in the cases of $N\geq2$, we have the following spectral gap inequality:
\begin{eqnarray}\label{eqq10018}
\|\rho\|_{H^1(\mathcal{C})}^2\geq(p+\ve)\int_{\mathcal{C}}\Psi^{p-1}\rho^2d\mu
\end{eqnarray}
for all $\rho\in\mathcal{M}^{\perp}$
where $\ve>0$ is a fixed small constant and
\begin{eqnarray*}
\mathcal{M}^{\perp}=\{\rho\in H^1(\mathcal{C})\mid \langle \rho,\Psi \rangle_{H^1(\mathcal{C})}=\langle \rho,\Psi' \rangle_{H^1(\mathcal{C})}=0\}.
\end{eqnarray*}
In this section, we shall improve the spectral gap inequality~\eqref{eqq10018} by proving the following result.
\begin{proposition}\label{propq0001}
Let $N\geq2$ and assume that either
\begin{enumerate}
\item[$(1)$]\quad $b_{FS}(a)< b<a+1$ with $a<0$ or
\item[$(2)$]\quad $a\leq b<a+1$ with $0\leq a<a_c$ and $a+b>0$.
\end{enumerate}
Then for every $\rho\in\mathcal{M}^{\perp}$, we have
\begin{eqnarray}\label{eqq0021}
\|\rho\|_{H^1(\mathcal{C})}^2-\beta\int_{\mathcal{C}}\cosh^{-2}(\gamma t)\rho^2d\mu\geq\lambda_*\|\rho\|_{H^1(\mathcal{C})}^2,
\end{eqnarray}
where 
\begin{eqnarray*}
\lambda_*=\left\{\aligned
&\frac{2(p-1)}{3p-1}, \quad a_c^*<a<a_c \text{ and } a\leq b<a+1,\\
&\frac{2(p-1)}{3p-1}, \quad a\leq a_c^*\text{ and }b_{FS}^*(a)\leq b<a+1,\\
&\frac{2q(a)-(p-2)(p+1)+(p-1)(1+q(a))^{\frac12}}{2+2q(a)+(p-1)(1+q(a))^{\frac12}},\quad a\leq a_c^*\text{ and }b_{FS}(a)<b<b_{FS}^*(a)
\endaligned\right.
\end{eqnarray*}
with
\begin{eqnarray*}
\beta=\frac{p(p+1)(a_c-a)^2}{2},\quad\gamma=\frac{(p-1)(a_c-a)}{2},\quad a_c^*=\bigg(1-\sqrt{\frac{N-1}{2N}}\bigg)a_c,
\end{eqnarray*}
and
\begin{eqnarray*}
b_{FS}^*(a)=\frac{(a_c-a)N}{a_c-a+\sqrt{(a_c-a)^2+N-1}}+a-a_c, \quad q(a)=\frac{N-1}{(a_c-a)^2}.
\end{eqnarray*}
Moreover, the equality of \eqref{eqq0021} holds if and only if for 
\begin{eqnarray}\label{eqq1022}
\rho_{0,2}(t)=\mathcal{P}(\tanh(\gamma t))(\cosh(\gamma t))^{-\frac{2}{p-1}}
\end{eqnarray}
with 
\begin{eqnarray*}
\mathcal{P}(z)=\frac{1}{8}(1-z^2)^{-\frac{2}{p-1}}\frac{d^2}{dz^2}((1-z^2)^{2+\frac{2}{p-1}})
\end{eqnarray*}
being a Jacobi polynomial under the assumptions
\begin{enumerate}
\item[$(i)$]\quad $a_c^*\leq a<a_c$ and $a\leq b<a+1$,
\item[$(ii)$]\quad $a<a_c^*$ and $b_{FS}^*(a)<b<a+1$,
\end{enumerate}
the equality of \eqref{eqq0021} holds if and only if for
\begin{eqnarray}\label{eqq3022}
\rho_{1,0,l}(t,\theta)=(\cosh(\gamma t))^{-\frac{\sqrt{(a_c-a)^2+N-1}}{\gamma}}\Theta_{1,l},\quad l=1,2,\cdots,N,
\end{eqnarray}
under the assumption $a<a_c^*$ and $b_{FS}(a)<b<b_{FS}^*(a)$, and the equality of \eqref{eqq0021} holds if and only if
either for $\rho_{0,2}(t)$ or for $\rho_{1,0,l}(t,\theta)$, $l=1,2,\cdots,N$, under the assumption $a<a_c^*$ and $b=b_{FS}^*(a)$,
where $\Theta_{1,l}$ are spherical harmonics of degree one.
\end{proposition}
\begin{proof}
Since by \eqref{eq0026}, $\Psi(t)\to0$ as $|t|\to+\infty$, it is well known that the operator $-\Delta_{\mathbb{S}^{N-1}}-\partial_t^2+(a_c-a)^2$ is compact in $L^2(p\Psi^{p-1}, \mathcal{C})$.  Thus, it is also well known that $\sigma(-\Delta_{\mathbb{S}^{N-1}}-\partial_t^2+(a_c-a)^2)=\{\lambda_l\}_{l\in\bbn}$ with $0<\lambda_1\leq\lambda_2\leq\cdots\leq\lambda_l\to+\infty$ as $l\to\infty$, where $\sigma(-\Delta_{\mathbb{S}^{N-1}}-\partial_t^2+(a_c-a)^2)$ is the spectrum of the operator $-\Delta_{\mathbb{S}^{N-1}}-\partial_t^2+(a_c-a)^2$ in $L^2(p\Psi^{p-1}, \mathcal{C})$.  

\vskip0.12in

Let us now consider the following eigenvalue problem
\begin{eqnarray}\label{eqq0016}
-\Delta_{\mathbb{S}^{N-1}}v-\partial_t^2v+(a_c-a)^2v=\lambda p\Psi^{p-1}v, \quad\text{in }\mathcal{C}.
\end{eqnarray}
where $\lambda>0$.  As usual, since the spherical harmonics $\{\Theta_{i,l}\}_{i\in\bbn_0,1\leq l\leq l_{i,N}}$, which satisfies the following equation
\begin{eqnarray}\label{eqq10016}
-\Delta_{\mathbb{S}^{N-1}}\Theta_{i,l}=i(N-2+i)\Theta_{i,l}\quad\text{in }\mathbb{S}^{N-1},
\end{eqnarray}
form a orthogonal basic of $L^2(\mathbb{S}^{N-1})$, we shall use $\{\Theta_{i,l}\}_{i\in\bbn_0,1\leq l\leq l_{i,N}}$ as the Fourier modes to expand the eigenvalue problem~\eqref{eq0016}, where $l_{i,N}\in\bbn$.  Since we have
\begin{eqnarray*}
v=\sum_{i=0}^{\infty}\sum_{l=1}^{l_{i,N}}\phi_{i,l}\Theta_{i,l}
\end{eqnarray*}
for every $v\in L^2(\mathcal{C})$ with $\phi_{i,l}=\int_{\mathcal{C}}v\Theta_{i,l}d\theta$,
by \eqref{eqq0016} and \eqref{eqq10016}, $v$ is a solution of the eigenvalue problem~\eqref{eqq0016} if and only if $\phi_{i,l}$ satisfies the following ordinary differential equation
\begin{eqnarray}\label{eqq0017}
-\partial_t^2\phi_{i,l}-\lambda\beta\cosh^{-2}(\gamma t)\phi_{i,l}=-\tau_{a,i}\phi_{i,l},\quad \text{in }\bbr
\end{eqnarray}
for all $l=1,2,\cdots,l_{i,N}$ and $i\in\bbn_0$, where 
\begin{eqnarray}\label{eqq5699}
\tau_{a,i}=(a_c-a)^2+i(N-2+i).
\end{eqnarray}

\vskip0.12in

By \cite[p. 74]{LL1958} (see also \cite[4.2.2. Example: Poschl-Teller potentials]{FLW2022} or \cite[p. 130]{FS2003}), the negative eigenvalues of the opreator $-\partial_t^2-\lambda\beta\cosh^{-2}(\gamma t)$ in $L^2(\bbr)$ is given by  
\begin{eqnarray*}
\sigma_j=-\frac{\gamma^2}{4}\bigg(-(2j+1)+\sqrt{1+4\lambda\beta\gamma^{-2}}\bigg)^2
\end{eqnarray*}
where $j=0,1,2,\cdots,j_0$ with $j_0\in\bbn_0$ and $j_0\leq\frac{1}{2}\bigg(\sqrt{1+4\lambda\beta\gamma^{-2}}-1\bigg)$.  It follows that the ordinary differential equation~\eqref{eqq0017} is solvable if and only if
\begin{eqnarray}\label{eqq0018}
\frac{\gamma^2}{4}\bigg(-(2j+1)+\sqrt{1+4\lambda\beta\gamma^{-2}}\bigg)^2=\tau_{a,i}.
\end{eqnarray}
For every $i$ and $j$, we denote the unique number of $\lambda>0$ which satisfies \eqref{eqq0018} by $\lambda_{i,j}$.  Thus, all eigenvalues of \eqref{eqq0016} are $\{\lambda_{i,j}\}_{i,j\in\bbn_0}$.  Since \eqref{eqq0016} with $\lambda=1$ is just \eqref{eq0016}, it has been proved in \cite{FS2003} that $\lambda_{0,1}=1$.  
Note that by \eqref{eqq0018}, $\lambda_{j,i}<\lambda_{j,i+1}$ and $\lambda_{j,i}<\lambda_{j+1,i}$ for all $i$ and $j$, thus, by $\lambda_{0,1}=1$, we have $\lambda_{0,0}<1$ and $1<\lambda_{i,j}$ for all other $i$ and $j$ except $\lambda_{1,0}$.  Moreover, since $\Psi$ has Morse index $1$ under the conditions~$(1)$ and $(2)$ in the cases of $N\geq2$, we must have $1<\lambda_{1,0}$.  Thus, $\min\{\lambda_{0,2}, \lambda_{1,1}, \lambda_{1,0}\}$ is the smallest eigenvalue of \eqref{eqq0016} which is larger than $1$.  

\vskip0.12in

Let us first compare $\lambda_{0,2}$ and $\lambda_{1,1}$.
We define
\begin{eqnarray*}
f(\lambda)=\sqrt{1+4\lambda\beta\gamma^{-2}}.
\end{eqnarray*}
Then by \eqref{eqq0018}, 
\begin{eqnarray*}
f(\lambda_{0,2})-f(\lambda_{1,1})=2(1-g_N(a)h^{-1}_{N,a}(b)),
\end{eqnarray*}
where $g_N(a)=\sqrt{\frac{1}{4}+\frac{N-1}{4(a_c-a)^2}}-\frac{1}{2}$ and $h_{N,a}(b)=\frac{1+a-b}{N-2(1+a-b)}$.  By direct calculations, we find that $h_{N,a}(b)$ is decreasing for $b$ with 
\begin{eqnarray*}
\min_{a\leq b\leq a+1}h_{N,a}(b)=h_{N,a}(a+1)=0,\quad\max_{a\leq b\leq a+1}h_{N,a}(b)=h_{N,a}(a)=\bigg(\frac{1}{N-2}\bigg)_+
\end{eqnarray*}
and $g_{N}(a)$ is increasing for $a$ with 
\begin{eqnarray*}
\min_{-\infty<a\leq \frac{N-2}{2}}g_{N}(a)=g_{N}(-\infty)=0,\quad\max_{-\infty<a\leq \frac{N-2}{2}}g_{N}(a)=g_{N}(\frac{N-2}{2})=+\infty.
\end{eqnarray*}
Here, $\bigg(\frac{1}{N-2}\bigg)_+=\frac{1}{N-2}$ for $N\geq3$ and $\bigg(\frac{1}{N-2}\bigg)_+=+\infty$ for $N=2$.
Note that $g_N(0)=\bigg(\frac{1}{N-2}\bigg)_+$, thus, for $0\leq a<\frac{N-2}{2}$ which implies $N\geq3$, we always have $g_N(a)h^{-1}_{N,a}(b)>1$ for all $0\leq a<\frac{N-2}{2}$ and $a\leq b<a+1$ with $a+b>0$, which implies that $f(\lambda_{0,2})<f(\lambda_{1,1})$ for all $0\leq a<\frac{N-2}{2}$ and $a\leq b<a+1$ with $a+b>0$.  Moreover, for every $a<0$, there exists a unique $a<b_{N,a}<a+1$ such that $g_N(a)h^{-1}_{N,a}(b_{N,a})=1$, which implies that $b_{N,a}=b_{FS}(a)$ given by \eqref{eqn991}.  Thus, by the monotone property of $h_{N,a}(b)$, we see that $g_N(a)h^{-1}_{N,a}(b)>1$ for all $a<0$ and $b_{FS}(a)<b<a+1$, which implies that $f(\lambda_{0,2})<f(\lambda_{1,1})$ for all $a<0$ and $b_{FS}(a)<b<a+1$.  It follows that we always have $\lambda_{0,2}<\lambda_{1,1}$ under the conditions~(1) and $(2)$ in the cases of $N\geq2$.  

\vskip0.12in

It is sufficiently to compare $\lambda_{0,2}$ and $\lambda_{1,0}$
to determine the smallest eigenvalue of \eqref{eqq0016} which is larger than $1$ under the conditions~(1) and $(2)$ in the cases of $N\geq2$.  As above, we have
\begin{eqnarray*}
f(\lambda_{0,2})-f(\lambda_{1,0})=2(2-g_N(a)h^{-1}_{N,a}(b)).
\end{eqnarray*}
Now, using the monotone properties of $g_N(a)$ and $h_{N,a}(b)$, we can compute as above to find that $g_N(a)h^{-1}_{N,a}(b)>2$ for $a_c^*< a<a_c$ with all $a\leq b<a+1$ in the cases of $N\geq2$, while for $a\leq a_c^*$ in the cases of $N\geq2$, $g_N(a)h^{-1}_{N,a}(b)>2$ for $b_{FS}^*(a)<b<a+1$, $g_N(a)h^{-1}_{N,a}(b)=2$ for $b=b_{FS}^*(a)$ and $g_N(a)h^{-1}_{N,a}(b)<2$ for $b_{FS}(a)<b<b_{FS}^*(a)$.  If follows that $\lambda_{0,2}<\lambda_{1,0}$ either for $a\leq a_c^*$ with $b_{FS}^*(a)<b<a+1$ or for $a_c^*<a<a_c$ with all $a\leq b<a+1$, $\lambda_{0,2}=\lambda_{1,0}$ for $a\leq a_c^*$ with $b=b_{FS}^*(a)$ and $\lambda_{1,0}<\lambda_{0,2}$ for $a<a_c^*$ with $b_{FS}(a)<b<b_{FS}^*(a)$, which, together with the fact that $\lambda_{0,2}<\lambda_{1,1}$ under the conditions~(1) and $(2)$ in the cases $N\geq2$, implies that $\lambda_{0,2}$ is the smallest eigenvalue of \eqref{eqq0016} which is larger than $1$ either for $a\leq a_c^*$ with $b_{FS}^*(a)<b<a+1$ or for $a_c^*<a<a_c$ with all $a\leq b<a+1$, $\lambda_{0,2}$ and $\lambda_{1,0}$ are both the smallest eigenvalue of \eqref{eqq0016} which is larger than $1$ for $a\leq a_c^*$ with $b=b_{FS}^*(a)$, and $\lambda_{1,0}$ is the smallest eigenvalue of \eqref{eqq0016} which is larger than $1$ for $a<a_c^*$ with $b_{FS}(a)<b<b_{FS}^*(a)$.

\vskip0.12in

Since $\lambda_{0,1}=1$, $\lambda_{0,0}<1$ and $1<\lambda_{i,j}$ for all other $i$ and $j$, we have
\begin{eqnarray}\label{eqq1021}
\|\rho\|_{H^1(\mathcal{C})}^2-\beta\int_{\mathcal{C}}\cosh^{-2}(\gamma t)\rho^2d\mu\geq\lambda_*\|\rho\|_{H^1(\mathcal{C})}^2\end{eqnarray}
for every $\rho\in\mathcal{M}^{\perp}$, where
\begin{eqnarray*}
\lambda_*=\left\{\aligned
&\frac{\lambda_{0,2}-1}{\lambda_{0,2}},\quad a_c^*<a<a_c\text{ with } a\leq b<a+1,\\
&\frac{\lambda_{0,2}-1}{\lambda_{0,2}},\quad a\leq a_c^*\text{ with }b_{FS}^*(a)< b<a+1,\\
&\frac{\lambda_{0,2}-1}{\lambda_{0,2}}=\frac{\lambda_{1,0}-1}{\lambda_{1,0}},\quad a\leq a_c^*\text{ with }b=b_{FS}^*(a),\\
&\frac{\lambda_{1,0}-1}{\lambda_{1,0}},\quad a\leq a_c^*\text{ with }b_{FS}(a)<b<b_{FS}^*(a).
\endaligned\right.
\end{eqnarray*}
By \eqref{eqq0018}, we can compute
\begin{eqnarray}\label{eqq1029}
\frac{\lambda_{0,2}-1}{\lambda_{0,2}}=\frac{(2(a_c-a)+4\gamma)\gamma}{(a_c-a+3\gamma)(a_c-a+2\gamma)}=\frac{2(p-1)}{3p-1}
\end{eqnarray}
and
\begin{eqnarray}\label{eqq3029}
\frac{\lambda_{1,0}-1}{\lambda_{1,0}}=\frac{\sqrt{(a_c-a)^2+N-1}(\sqrt{(a_c-a)^2+N-1}+\gamma)-\beta}{\sqrt{(a_c-a)^2+N-1}(\sqrt{(a_c-a)^2+N-1}+\gamma)}
\end{eqnarray}
which, together with \eqref{eqq1021}, implies that \eqref{eqq0021} holds true for every $\rho\in\mathcal{M}^{\perp}$ under the conditions~$(1)$ and $(2)$ in the cases of $N\geq2$.

\vskip0.12in

It remains to prove that the equality of \eqref{eqq0021} holds if and only if for the functions given by \eqref{eqq1022} and \eqref{eqq3022}.  By \cite[p. 129, Case~1]{GLKO1993}, we have 
\begin{eqnarray}\label{eqq0020}
\phi_{i,j}(t)=\chi_{k}(\sinh(\gamma t))(\cosh(\gamma t))^{-\frac{j\gamma+\sqrt{\tau_{a,i}}}{\gamma}},
\end{eqnarray}
where $\chi_k(z)$ is a polynomial of degree at most $k$ which depends on $i,j$.  Moreover, by \cite[Theorem~7]{GLKO1993}, $\chi_k(z)$ satisfies the following equation
\begin{eqnarray}\label{eqq1090}
-\gamma^2(z^2+1)D_z^2\chi_k(z)-((1-2j)\gamma^2-2\gamma\sqrt{\tau_{a,i}})D_z\chi_k(z)+R\chi_k(z)=0
\end{eqnarray}
where $R\in\bbr$ can be taken arbitrary values.  As that in \cite[p. 74]{LL1958} (see also \cite[p. 529]{S2007}), we introduce the function
\begin{eqnarray*}
\varphi_{j,k}(z)=(1+z^2)^{-\frac{j}{2}}\chi_k(z)\quad\text{and}\quad\widetilde{\varphi}_{j,k}(y)=\varphi_{j,k}\bigg(\frac{y}{\sqrt{1-y^2}}\bigg)
\end{eqnarray*}
with $z=\frac{y}{\sqrt{1-y^2}}$.  Then by direct calculations, $1+z^2=\frac{1}{1-y^2}$.  Moreover,
\begin{eqnarray*}
D_y\widetilde{\varphi}_{j,k}(y)=\frac{D_z\varphi_{j,k}(z)}{(1-y^2)^{\frac{3}{2}}}
\end{eqnarray*}
and
\begin{eqnarray*}
D_y^2\widetilde{\varphi}_{j,k}(y)=\frac{D_z^2\varphi_{j,k}(z)}{(1-y^2)^{3}}+\frac{3yD_z\varphi_{j,k}(z)}{(1-y^2)^{\frac{5}{2}}}
\end{eqnarray*}
with
\begin{eqnarray*}
D_z\varphi_{j,k}(z)=-j(1+z^2)^{-\frac{j+2}{2}}z\chi_k(z)+(1+z^2)^{-\frac{j}{2}}D_z\chi_k(z)
\end{eqnarray*}
and
\begin{eqnarray*}
D_z^2\varphi_{j,k}(z)&=&(1+z^2)^{-\frac{j}{2}}D_z^2\chi_k(z)-2j(1+z^2)^{-\frac{j+2}{2}}z\chi_k(z)\\
&&+j\bigg(\frac{(j+2)z^2}{1+z^2}-1\bigg)(1+z^2)^{-\frac{j+2}{2}}\chi_k(z).
\end{eqnarray*}
It follows that for every $\widetilde{\tau}, \widetilde{\mu}$ and $q$,
\begin{eqnarray*}
&&(1-y^2)D_y^2\widetilde{\varphi}_{j,k}(y)+(\widetilde{\tau}-\widetilde{\mu}-(\widetilde{\tau}+\widetilde{\mu}+2))D_y\widetilde{\varphi}_{j,k}(y)+q(\widetilde{\tau}+\widetilde{\mu}+q+1)\widetilde{\varphi}_{j,k}(y)\\
&=&(1+z^2)^{\frac{-j+4}{2}}D_z^2\chi_k(z)+((\widetilde{\tau}-\widetilde{\mu})(1+z^2)^{\frac{-j+3}{2}}+(1-2j-\widetilde{\tau}-\widetilde{\mu})(1+z^2)^{\frac{-j+2}{2}}z)D_z\chi_k(z)\\
&&+(-j(\widetilde{\tau}-\widetilde{\mu})(1+z^2)^{\frac{1}{2}}+(j^2+j(\widetilde{\tau}+\widetilde{\mu}))z^2+q(\widetilde{\tau}+\widetilde{\mu}+q+1)-j)(1+z^2)^{\frac{-j}{2}}\chi_k(z).
\end{eqnarray*}
By \eqref{eqq1090}, we find that $\widetilde{\varphi}_{j,k}(y)$ satisfies the following Jacobi equation
\begin{eqnarray*}
(1-y^2)D_y^2\widetilde{\varphi}_{j,k}(y)-2\bigg(\frac{\sqrt{\tau_{a,i}}}{\gamma}+1\bigg)yD_y\widetilde{\varphi}_{j,k}(y)+j\bigg(\frac{2\sqrt{\tau_{a,i}}}{\gamma}+j+1\bigg)\widetilde{\varphi}_{j,k}(y)=0.
\end{eqnarray*}
It follows from \cite[p. 22]{NU1988} that $\widetilde{\varphi}_{j,k}(y)$ is the Jacobi polynomial given by
\begin{eqnarray}\label{eqq0027}
\widetilde{\varphi}_{j,k}(y)=\frac{(-1)^j}{2^jj!}(1-y^2)^{-\frac{\sqrt{\tau_{a,i}}}{\gamma}}\frac{d^j}{dy^j}\bigg((1-y^2)^{j+\frac{\sqrt{\tau_{a,i}}}{\gamma}}\bigg).
\end{eqnarray}
Thus, by \eqref{eqq0021}, we know that the equality of \eqref{eqq0021} holds if and only if for the functions given by \eqref{eqq1022} and \eqref{eqq3022}.
\end{proof}

\begin{remark}\label{rmkq0001}
Since $\tanh(\gamma t)$ is bounded in $\bbr$, by \eqref{eqq0020} and  \eqref{eqq0027},
\begin{eqnarray*}
|\phi_{i,j}|\lesssim\Psi\quad\text{for all $i$ and $l$}.
\end{eqnarray*}
In particular, by \eqref{eqq1022} and \eqref{eqq3022}, we have $|\rho_{0,2}|\lesssim\Psi$ and $|\rho_{1,0,l}|\lesssim\Psi$ for all $l=1,2,\cdots,N$.
\end{remark}

\begin{remark}
By \eqref{eqq3022} and \eqref{eqq3029}, we find that $\rho_{1,0,l}(t,\theta)=(\cosh(\gamma t))^{\frac{p+1}{p-1}}\theta_{1,l}$ and $\lambda_{1,0}=1$ for $b=b_{FS}(a)$, which coincides with the computations in \cite[Lemma~7]{FP2023} (see also \cite[(2.10)]{FS2003}).
\end{remark}

\vskip0.2in

\section{Energy estimates of $c_{BE}$}
To prove $c_{BE}$ is achieved, we shall follow the ideas of Konig in \cite{K2023} to derive two crucial energy estimates of $c_{BE}$.  For this purpose, we need first to establish the following expression of $dist_{H^1(\mathcal{C})}^2(v, \mathcal{Y})$, where 
$\mathcal{Y}$ is given by \eqref{eqq1030}.
\begin{lemma}\label{lemq0001}
Let $N\geq2$ and assume that either
\begin{enumerate}
\item[$(1)$]\quad $b_{FS}(a)< b<a+1$ with $a<0$ or
\item[$(2)$]\quad $a\leq b<a+1$ with $0\leq a<a_c$ and $a+b>0$.
\end{enumerate}
Then for every $v\in H^1(\mathcal{C})$,
\begin{eqnarray}\label{eqq0031}
dist_{H^1(\mathcal{C})}^2(v, \mathcal{Y})=\|v\|^2_{H^1(\mathcal{C})}-C_{a,b,N}^{-1}\sup_{h\in\mathcal{Y}_1}(\langle v,h^{p} \rangle_{L^2(\mathcal{C})})^2
\end{eqnarray} 
where
\begin{eqnarray*}
\mathcal{Y}_1=\{v\in\mathcal{Y}\mid\|v\|_{L^{p+1}(\mathcal{C})}=1\}.
\end{eqnarray*}
Moreover, $\sup_{h\in\mathcal{Y}_1}(\langle v,h^{p} \rangle_{L^2(\mathcal{C})})^2$ is attained for every $v\in H^1(\mathcal{C})$.
\end{lemma}
\begin{proof}
The proof is a ``completion of the square''
argument which is the same as that of \cite[Lemma~3]{DEFFL2022} (see also the proof of \cite[Lemma~2.2]{K2023}), so we omit it here.
\end{proof}

For the convenience of the readers, we provide here a standard computation of the integral $\int_{\bbr}(\cosh(s))^{-\alpha}(\cosh^2(s)-1)^\beta ds$ which is also used in the appendix:
\begin{eqnarray}
\int_{\bbr}(\cosh(s))^{-\alpha}(\cosh^2(s)-1)^\beta ds&=&2\int_{0}^{+\infty}(\cosh(s))^{-\alpha}(\cosh^2(s)-1)^\beta ds\notag\\
&=&2\int_{0}^{+\infty}(1-\cosh^{-2}(s))^\beta(\cosh(s))^{2\beta-\alpha} ds\notag\\
&=&-\int_{0}^{+\infty}(1-\cosh^{-2}(s))^{\beta-\frac12}(\cosh(s))^{2\beta-\alpha+2}d(\cosh^{-2}(s))\notag\\
&=&\int_0^1(1-x)^{\beta-\frac12}x^{\frac{\alpha}{2}-\beta-1}dx\notag\\
&=&\mathbb{B}(\frac{\alpha}{2}-\beta,\beta+\frac12),\label{eqq1041}
\end{eqnarray}
for $\frac{\alpha}{2}>\beta$ and $\beta>-\frac12$.
Now, we have the following crucial energy estimates of $c_{BE}$.
\begin{proposition}\label{propq0002}
Let $N\geq2$ and assume that either
\begin{enumerate}
\item[$(1)$]\quad $b_{FS}(a)< b<a+1$ with $a<0$ or
\item[$(2)$]\quad $a\leq b<a+1$ with $0\leq a<a_c$ and $a+b>0$.
\end{enumerate}
Then $c_{BE}<2-2^{\frac{1}{p+1}}$.  Moreover, if either
\begin{enumerate}
\item[$(i)$]\quad $a_c^*\leq a<a_c$ and $a\leq b<a+1$ or
\item[$(ii)$]\quad $a<a_c^*$ and $b_{FS}^*(a)\leq b<a+1$.
\end{enumerate}
Then $c_{BE}<\frac{2(p-1)}{3p-1}$, where $a_c^*$ and $b_{FS}^*(a)$ are given in Proposition~\ref{propq0001}.
\end{proposition}
\begin{proof}
Let us first prove that
\begin{eqnarray}\label{eqq1034}
c_{BE}<\frac{2(p-1)}{3p-1}.
\end{eqnarray}
under the assumptions~$(i)$ and $(ii)$.  Testing $c_{BE}$ by the function $u=\Psi+\ve\rho_{0,2}$ with $\ve\to0$, then by the definition of $c_{BE}$, we have
\begin{eqnarray}\label{eqq0023}
c_{BE}\leq\frac{\|\Psi+\ve\rho_{0,2}\|^2_{H^1(\mathcal{C})}-C_{a,b,N}^{-1}\|\Psi+\ve\rho_{0,2}\|^2_{L^{p+1}(\mathcal{C})}}{dist_{H^1(\mathcal{C})}^2(\Psi+\ve\rho_{0,2}, \mathcal{Y})}.
\end{eqnarray}
where $\rho_{0,2}$ is given by \eqref{eqq1022}.  By $\rho_{0,2}\in\mathcal{M}^{\perp}$ and Lemma~\ref{lemq0001},
\begin{eqnarray*}
dist_{H^1(\mathcal{C})}^2(\Psi+\ve\rho_{0,2}, \mathcal{Y})=\ve^2\|\rho_{0,2}\|^2_{H^1(\mathcal{C})},
\end{eqnarray*}
which, together with \eqref{eqq0023}, implies that
\begin{eqnarray}\label{eqq1023}
c_{BE}\leq\frac{\|\Psi+\ve\rho_{0,2}\|^2_{H^1(\mathcal{C})}-C_{a,b,N}^{-1}\|\Psi+\ve\rho_{0,2}\|^2_{L^{p+1}(\mathcal{C})}}{\ve^2\|\rho_{0,2}\|^2_{H^1(\mathcal{C})}}.
\end{eqnarray}
By Remark~\ref{rmkq0001}, we can expand $\|\Psi+\ve\rho_{0,2}\|^{p+1}_{L^{p+1}(\mathcal{C})}$ by the Taylor expansion to arbitrary order terms.  Thus,
\begin{eqnarray*}
\|\Psi+\ve\rho_{0,2}\|^{p+1}_{L^{p+1}(\mathcal{C})}&=&\|\Psi\|^{p+1}_{L^{p+1}(\mathcal{C})}+\ve(p+1)\langle \Psi^{p}, \rho_{0,2}\rangle_{L^2(\mathcal{C})}+\ve^2\frac{(p+1)p}{2}\langle \Psi^{p-1}, \rho_{0,2}^2\rangle_{L^2(\mathcal{C})}\\
&&+\ve^3\frac{p(p^2-1)}{6}\langle \Psi^{p-2}, \rho_{0,2}^3\rangle_{L^2(\mathcal{C})}+o(\ve^3).
\end{eqnarray*}
It follows from \eqref{eq0006}, $\rho_{0,2}\in\mathcal{M}^{\perp}$ and the Taylor expansion once more, that
\begin{eqnarray}
C_{a,b,N}^{-1}\|\Psi+\ve\rho_{0,2}\|^2_{L^{p+1}(\mathcal{C})}
&=&C_{a,b,N}^{-1}\|\Psi\|^2_{L^{p+1}(\mathcal{C})}+p\ve^2\langle \Psi^{p-1}, \rho_{0,2}^2\rangle_{L^2(\mathcal{C})}\notag\\
&&+\frac{p(p-1)\ve^3}{3}\langle \Psi^{p-2}, \rho_{0,2}^3\rangle_{L^2(\mathcal{C})}
+o(\ve^3).\label{eqq1033}
\end{eqnarray}
By Proposition~\ref{propq0001},
\begin{eqnarray*}
\rho_{0,2}=\frac{p(\cosh(\gamma s))^{-\frac{2}{p-1}}}{4(p-1)^2}(4(p+1)-(6p+2)(\cosh(\gamma s))^{-2})
\end{eqnarray*}
under the assumptions~$(i)$ and $(ii)$.  It follows from \eqref{eq0026} and \eqref{eqq1041} that 
\begin{eqnarray*}
\langle \Psi^{p-2}, \rho_{0,2}^3\rangle_{L^2(\mathcal{C})}&=&\bigg(\frac{(p+1)(a_c-a)^2}{2}\bigg)^{\frac{p-2}{p-1}}|\mathbb{S}^{N-1}|\int_{\bbr}(\cosh(\gamma s))^{-\frac{2(p-2)}{p-1}}\rho_{0,2}^3ds\\
&=&\frac{p^3}{8\gamma(p-1)^6}\bigg(\frac{(p+1)(a_c-a)^2}{2}\bigg)^{\frac{p-2}{p-1}}|\mathbb{S}^{N-1}|\bigg(8(p+1)^3\mathbb{B}(\frac{p+2}{p-1},\frac12)\\
&&-12(p+1)^2(3p+1)\mathbb{B}(\frac{p+2}{p-1}+1,\frac12)+6(p+1)(3p+1)^2\mathbb{B}(\frac{p+2}{p-1}+2,\frac12)\\
&&-(3p+1)^3\mathbb{B}(\frac{p+2}{p-1}+3,\frac12)\bigg),
\end{eqnarray*}
which, together with the well known fact that $\mathbb{B}(m,n)=\frac{m-1}{m-1+n}\mathbb{B}(m-1,n)$, implies that
\begin{eqnarray*}
\langle \Psi^{p-2}, \rho_{0,2}^3\rangle_{L^2(\mathcal{C})}&=&\frac{2(p+1)p^3}{\gamma(7p-3)(5p-1)(p-1)^6}\bigg(\frac{(p+1)(a_c-a)^2}{2}\bigg)^{\frac{p-2}{p-1}}|\mathbb{S}^{N-1}|\\
&&\times\mathbb{B}(\frac{p+2}{p-1},\frac12)(p^4-6p^2+8p-3)\\
&>&0
\end{eqnarray*}
since $p>1$.
Thus, by \eqref{eq0006}, \eqref{eq0009}, \eqref{eqq1023}, \eqref{eqq1033} and Propostion~\ref{propq0001},
\begin{eqnarray*}
c_{BE}&\leq&\frac{\|\rho_{0,2}\|^2_{H^1(\mathcal{C})}-p\langle \Psi^{p-1}, \rho_{0,2}^2\rangle_{L^2(\mathcal{C})}}{\|\rho_{0,2}\|^2_{H^1(\mathcal{C})}}-\ve\frac{p(p-1)\langle \Psi^{p-2}, \rho_{0,2}^3\rangle_{L^2(\mathcal{C})}}{3\|\rho_{0,2}\|^2_{H^1(\mathcal{C})}}+o(\ve)\\
&<&\frac{\|\rho_{0,2}\|^2_{H^1(\mathcal{C})}-p\langle \Psi^{p-1}, \rho_{0,2}^2\rangle_{L^2(\mathcal{C})}}{\|\rho_{0,2}\|^2_{H^1(\mathcal{C})}}\\
&=&\frac{2(p-1)}{3p-1}
\end{eqnarray*}
for $\ve>0$ sufficiently small, which implies that \eqref{eqq1034} holds true under the assumptions~$(i)$ and $(ii)$.

\vskip0.12in

It remains to prove that $c_{BE}<2-2^{\frac{2}{p+1}}$ for $N\geq2$ under the conditions~$(1)$ and $(2)$.  For this purpose, we use $v_s=\Psi+\Psi_{s}$ as a test function of $c_{BE}$, where $s\to+\infty$.  Then we have
\begin{eqnarray}\label{eqq1035}
c_{BE}\leq\frac{\|v_s\|^2_{H^1(\mathcal{C})}-C_{a,b,N}^{-1}\|v_s\|^2_{L^{p+1}(\mathcal{C})}}{dist_{H^1(\mathcal{C})}^2(v_s, \mathcal{Y})}.
\end{eqnarray}
By \eqref{eq0026}, \eqref{eq0006} and direct calculations, we have
\begin{eqnarray}\label{eqq0040}
\|v_s\|_{H^1(\mathcal{C})}^2&=&2\|\Psi\|_{H^1(\mathcal{C})}^2+2\int_{\mathcal{C}}\Psi^p\Psi_sd\mu\notag\\
&=&2\|\Psi\|_{H^1(\mathcal{C})}^2+2\int_{\{t<\frac{s}{2}\}\times\mathbb{S}^{N-1}}(\Psi^p\Psi_s+\Psi_s^p\Psi)d\mu\notag\\
&=&2\|\Psi\|_{H^1(\mathcal{C})}^2+2A_0e^{-\frac{2}{p-1}\gamma s}+\mathcal{O}(e^{-\frac{p+1}{p-1}\gamma s}),
\end{eqnarray}
where
\begin{eqnarray*}
A_0=\bigg(\frac{(p+1)(a_c-a)^2}{2}\bigg)^{\frac{p+1}{p-1}}\int_{\mathcal{C}}(\cosh(\gamma t))^{-\frac{2p}{p-1}}e^{\frac{2}{p-1}\gamma t}dt.
\end{eqnarray*}
Since $\Psi_s\leq\Psi$ in $(-\infty, \frac{s}{2})$ by \eqref{eq0026}, by \eqref{eq0026} once more and the Taylor expansion,
\begin{eqnarray*}
\|v_s\|_{L^{p+1}(\mathcal{C})}^{p+1}&=&2\int_{\{t<\frac{s}{2}\}\times\mathbb{S}^{N-1}}(\Psi_s+\Psi)^{p+1}d\mu\\
&=&2\int_{\{t<\frac{s}{2}\}\times\mathbb{S}^{N-1}}\Psi^{p+1}d\mu+2(p+1)\int_{\{t<\frac{s}{2}\}\times\mathbb{S}^{N-1}}\Psi^{p}\Psi_sd\mu+\int_{\{t<\frac{s}{2}\}\times\mathbb{S}^{N-1}}\mathcal{O}(\Psi^{p-1}\Psi_s^2)d\mu\\
&=&2\|\Psi\|_{L^{p+1}(\mathcal{C})}^{p+1}+2(p+1)A_0e^{-\frac{2}{p-1}\gamma s}+o(e^{-\frac{2}{p-1}\gamma s}).
\end{eqnarray*}
Thus, by \eqref{eq0006}, \eqref{eq0009} and the Taylor expasion, we have
\begin{eqnarray}\label{eqq1036}
\|v_s\|_{H^1(\mathcal{C})}^2-C_{a,b,N}^{-1}\|v_s\|_{L^{p+1}(\mathcal{C})}^{2}=(2-2^{\frac{2}{p+1}})\|\Psi\|_{H^1(\mathcal{C})}^2-2A_0e^{-\frac{2}{p-1}\gamma s}+o(e^{-\frac{2}{p-1}\gamma s}).
\end{eqnarray}
On the other hand, by \eqref{eq0006} and Lemma~\ref{lemq0001},
\begin{eqnarray}
dist_{H^1(\mathcal{C})}^2(v_s, \mathcal{Y})&=&\|v_s\|^2_{H^1(\mathcal{C})}-C_{a,b,N}^{-1}\sup_{h\in\mathcal{Y}_1}(\langle v_s,h^{p} \rangle_{L^2(\mathcal{C})})^2\notag\\
&=&\|v_s\|^2_{H^1(\mathcal{C})}-C_{a,b,N}^{-1+\frac{2p}{p-1}}\sup_{\tau\in\bbr}(\langle \Psi+\Psi_s,\Psi_{\tau}^{p} \rangle_{L^2(\mathcal{C})})^2\label{eqq0035}
\end{eqnarray}
We denote $H_s(\tau)=F(\tau)+G_s(\tau)$ with 
\begin{eqnarray*}
F(\tau)=\langle \Psi,\Psi_{\tau}^{p} \rangle_{L^2(\mathcal{C})}\quad\text{and}\quad G_s(\tau)=\langle\Psi_{s},\Psi_{\tau}^{p} \rangle_{L^2(\mathcal{C})}.
\end{eqnarray*}
Clearly, by \eqref{eq0026} and the symmetry of $\Psi$,
\begin{eqnarray*}
\sup_{\tau\in\bbr}H_s(\tau)^2=\max_{0\leq\tau\leq\frac{s}{2}}H_s(\tau)^2=(\max_{0\leq\tau\leq\frac{s}{2}}H_s(\tau))^2.
\end{eqnarray*}
Moreover, $H_s(\tau)$ is strictly increasing in $(-\infty, 0)$ and strictly decreasing in $(s, +\infty)$.  We denote 
\begin{eqnarray*}
H_s(\tau(s))=\max_{0\leq\tau\leq\frac{s}{2}}H_s(\tau).
\end{eqnarray*}
Note that by \eqref{eq0026} and the symmetry of $\Psi$,  $F(\tau)$ is also strictly increasing in $(-\infty, 0)$ and strictly decreasing in $(0, +\infty)$.  Thus, $F(0)$ is the unique strictly global maximum of $F(\tau)$.  Since we also have
\begin{eqnarray*}
F(\tau)=\langle \Psi_{-\tau},\Psi^{p} \rangle_{L^2(\mathcal{C})}\quad\text{and}\quad G_s(\tau)=\langle\Psi_{s-\tau},\Psi^{p} \rangle_{L^2(\mathcal{C})},
\end{eqnarray*}
by similar estimates of \eqref{eqq0040},
\begin{eqnarray*}
F(0)+o(1)=H_{s}(0)\leq H_s(\tau(s))=F(\tau(s))+o(1)\leq F(0)+o(1)
\end{eqnarray*}
as $s\to+\infty$.
It follows from the continuity and monotone property of $F(\tau)$ that $\tau(s)=o(1)$ as $s\to+\infty$.  Again, by similar estimates of \eqref{eqq0040} and the Taylor expansion,
\begin{eqnarray*}
H_s(\tau(s))&=&F(\tau(s))+G_s(\tau(s))\\
&=&F(0)+\frac{F''(0)}{2}\tau(s)^2+o(\tau(s)^2)+2A_0e^{-\frac{2}{p-1}\gamma (s-\tau(s))}+\mathcal{O}(e^{-\frac{p+1}{p-1}\gamma s}),
\end{eqnarray*}
which, together with $H_s(\tau(s))\geq H(0)=F(0)+G_s(0)$ and $F''(0)<0$, implies that
\begin{eqnarray*}
\tau(s)^2&\lesssim&e^{-\frac{2}{p-1}\gamma (s-\tau(s))}-e^{-\frac{2}{p-1}\gamma s}+\mathcal{O}(e^{-\frac{p+1}{p-1}\gamma s})\\
&\lesssim&\tau(s)e^{-\frac{2}{p-1}\gamma s}+o(\tau(s)^2)+\mathcal{O}(e^{-\frac{p+1}{p-1}\gamma s}).
\end{eqnarray*}
It follows from $p>1$ that $\tau(s)=o(e^{-\frac{1}{p-1}\gamma s})$.  Thus, by the Taylor expansion and similar estimates of \eqref{eqq0040} once more, we have
\begin{eqnarray*}
H_s(\tau(s))&=&F(0)+G_s(0)+\mathcal{O}(\tau(s)^2)+G_s(\tau(s))-G(0)\\
&=&F(0)+G_s(0)++\mathcal{O}(\tau(s)^2)+\mathcal{O}(\tau(s)e^{-\frac{2}{p-1}\gamma s}+e^{-\frac{p+1}{p-1}\gamma s})\\
&=&F(0)+G_s(0)+o(e^{-\frac{2}{p-1}\gamma s}).
\end{eqnarray*}
By \eqref{eq0006}, \eqref{eq0009}, \eqref{eqq0040} and \eqref{eqq0035},
\begin{eqnarray}
dist_{H^1(\mathcal{C})}^2(v_s, \mathcal{Y})&=&2\|\Psi\|_{H^1(\mathcal{C})}^2+2\int_{\mathcal{C}}\Psi^p\Psi_sd\mu-C_{a,b,N}^{-1+\frac{2p}{p-1}}\bigg(\|\Psi\|_{L^{p+1}(\mathcal{C})}^{p+1}+\int_{\mathcal{C}}\Psi^p\Psi_sd\mu\bigg)^{2}\notag\\
&&+o(e^{-\frac{2}{p-1}\gamma s})\notag\\
&=&\|\Psi\|_{H^1(\mathcal{C})}^2-C_{a,b,N}^{-1+\frac{2p}{p-1}}(\int_{\mathcal{C}}\Psi^p\Psi_sd\mu)^2+o(e^{-\frac{2}{p-1}\gamma s})\notag\\
&=&\|\Psi\|_{H^1(\mathcal{C})}^2+o(e^{-\frac{2}{p-1}\gamma s}).\label{eqq1037}
\end{eqnarray}
By \eqref{eqq1035}, \eqref{eqq1036} and \eqref{eqq1037}, we have
\begin{eqnarray*}
c_{BE}&\leq&\frac{\|v_s\|^2_{H^1(\mathcal{C})}-C_{a,b,N}^{-1}\|v_s\|^2_{L^{p+1}(\mathcal{C})}}{dist_{H^1(\mathcal{C})}^2(v_s, \mathcal{Y})}\\
&=&2-2^{\frac{2}{p+1}}-2A_0C_{a,b,N}^{\frac{2}{p-1}}e^{-\frac{2}{p-1}\gamma s}+o(e^{-\frac{2}{p-1}\gamma s})\\
&<&2-2^{\frac{2}{p+1}}
\end{eqnarray*}
for $s>0$ sufficiently large, which completes the proof.
\end{proof}

\vskip0.2in

\section{Proof of main results}
We mainly follow the strategy of Konig in \cite{K2023} to prove Theorem~\ref{thmq0001}.

\vskip0.12in

\noindent\textbf{Proof of Theorem~\ref{thmq0001}:} 
Let $v_n$ be a minimizing suquence of \eqref{eqq0001}.  Then we have $\{v_n\}\subset H^1(\mathcal{C})\backslash\mathcal{Y}$ and
\begin{eqnarray}\label{eqq0030}
\frac{\|v_n\|^2_{H^1(\mathcal{C})}-C_{a,b,N}^{-1}\|v_n\|^2_{L^{p+1}(\mathcal{C})}}{dist_{H^1(\mathcal{C})}^2(v_n, \mathcal{Y})}=c_{BE}+o_n(1).
\end{eqnarray}
As that in \cite{K2023}, we normlize $v_n$ by assuming $\|v_n\|^2_{L^{p+1}(\mathcal{C})}=1$.  It follows from \eqref{eqq0030} that
\begin{eqnarray}\label{eq0002}
(c_{BE}+o_n(1))dist_{H^1(\mathcal{C})}^2(v_n, \mathcal{Y})+C_{a,b,N}^{-1}=\|v_n\|^2_{H^1(\mathcal{C})}.
\end{eqnarray}
Since by Proposition~\ref{propq0002}, \eqref{eqq1029} and \eqref{eqq3029}, $0<c_{BE}<1$ under the conditions $(1)$ and $(2)$, by Lemma~\ref{lemq0001}, 
\begin{eqnarray*}
(1-c_{BE}+o_n(1))\|v_n\|^2_{H^1(\mathcal{C})}=C_{a,b,N}^{-1}-(c_{BE}+o_n(1))C_{a,b,N}^{-1}\sup_{h\in\mathcal{Y}_1}(\langle v,h^{p} \rangle_{L^2(\mathcal{C})})^2.
\end{eqnarray*}
Thus, it is easy to see that $\{v_n\}$ is bounded in $H^1(\mathcal{C})$, which together with Lemma~\ref{lemq0001}, also implies that $\{dist_{H^1(\mathcal{C})}^2(v_n, \mathcal{Y})\}$ is bounded.  As that in \cite{K2023}, by the Lions lemma (cf. \cite[Lemma~4.1]{CW2001}), up to translating the sequence $\{v_n\}$, we may assume that $v_n\rightharpoonup f$ weakly in $H^1(\mathcal{C})$ for some non-zero $f$.  We decompose
\begin{eqnarray*}
v_n=f+g_n\quad \text{in $H^1(\mathcal{C})$ where $g_n\rightharpoonup0$ weakly in $H^1(\mathcal{C})$}. 
\end{eqnarray*}
For the sake of clarity, we divide the following proof into three steps.

{\bf Step.~1}\quad We prove that $g_n\to0$ strongly in $H^1(\mathcal{C})$ as $n\to\infty$.

Suppose the contrary that $g_n\not\to0$ strongly in $H^1(\mathcal{C})$, then by the Lions lemma (cf. \cite[Lemma~4.1]{CW2001}) once more and the fact that $g_n\rightharpoonup0$ weakly in $H^1(\mathcal{C})$, there exist $s_n\in\bbr$ such that $|s_n|\to+\infty$ and  
$g_n(\cdot-s_n)\rightharpoonup g_0\not=0$ weakly in $H^1(\mathcal{C})$.  We denote
$\mathbb{M}(v)=\sup_{h\in\mathcal{Y}_1}(\langle v,h^{p} \rangle_{L^2(\mathcal{C})})^2$.  Then by Lemma~\ref{lemq0001},
\begin{eqnarray*}
\mathbb{M}(g_n)=(\langle g_n,h_{n,*}^{p} \rangle_{L^2(\mathcal{C})})^2\quad\text{and}\quad\mathbb{M}(f)=(\langle f,h_f^{p} \rangle_{L^2(\mathcal{C})})^2.
\end{eqnarray*}
Since $g_n(\cdot-s_n)\rightharpoonup g_0\not=0$ weakly in $H^1(\mathcal{C})$ with $|s_n|\to+\infty$, we must have $h_{n,*}=C_{a,b,N}^{\frac{1}{p-1}}\Psi(t-s_{n,*}')$ with $|s_{n,*}'|\to+\infty$.
It follows that
\begin{eqnarray*}
\mathbb{M}(v_n)\geq(\langle v_n,h_{n,*}^{p} \rangle_{L^2(\mathcal{C})})^2=\mathbb{M}(g_n)+o_n(1)
\end{eqnarray*}
and
\begin{eqnarray*}
\mathbb{M}(v_n)\geq(\langle v_n,h_{f}^{p} \rangle_{L^2(\mathcal{C})})^2=\mathbb{M}(f)+o_n(1).
\end{eqnarray*}
Thus,
\begin{eqnarray}\label{eqq2036}
\mathbb{M}(v_n)\geq\max\bigg\{\mathbb{M}(g_n), \mathbb{M}(f)\bigg\}+o_n(1).
\end{eqnarray}
We denote $\mathcal{S}(v)=\frac{\|v\|^2_{H^1(\mathcal{C})}}{\|v\|^2_{L^{p+1}(\mathcal{C})}}$.  Moreover, without loss of generality, we assume that $\|g_n\|^2_{L^{p+1}(\mathcal{C})}\leq\|f\|^2_{L^{p+1}(\mathcal{C})}$.  Then by \eqref{eqq0001}, \eqref{eqq0031}, \eqref{eq0002}, \eqref{eqq2036} and the fact that $c_{BE}<1$,
\begin{eqnarray}
o_n(1)&=&o_n(1)dist_{H^1(\mathcal{C})}^2(v_n, \mathcal{Y})\notag\\
&=&\|v_n\|^2_{H^1(\mathcal{C})}-C_{a,b,N}^{-1}-c_{BE}dist_{H^1(\mathcal{C})}^2(v_n, \mathcal{Y})\notag\\
&=&(1-c_{BE})\|v_n\|^2_{H^1(\mathcal{C})}-C_{a,b,N}^{-1}+c_{BE}C_{a,b,N}^{-1}\sup_{h\in\mathcal{Y}_1}(\langle v_n,h^{p} \rangle_{L^2(\mathcal{C})})^2\notag\\
&\geq&\|f\|^2_{H^1(\mathcal{C})}-C_{a,b,N}^{-1}\|f\|^2_{L^{p+1}(\mathcal{C})}-c_{BE}dist_{H^1(\mathcal{C})}^2(f, \mathcal{Y})+(1-c_{BE})\|g_n\|^2_{H^1(\mathcal{C})}\notag\\
&&-C_{a,b,N}^{-1}\bigg((\|f\|_{L^{p+1}(\mathcal{C})}^{p+1}+\|g_n\|_{L^{p+1}(\mathcal{C})}^{p+1})^{\frac{2}{p+1}}-\|f\|^2_{L^{p+1}(\mathcal{C})}\bigg)+o_n(1)\notag\\
&\geq&\bigg(1-c_{BE}-\frac{C_{a,b,N}^{-1}}{\mathcal{S}(g_n)}\bigg(\frac{(q_n^{p+1}+1)^{\frac{2}{p+1}}-1}{q_n^2}\bigg)\bigg)\|g_n\|^2_{H^1(\mathcal{C})}+o(1),\label{eqq1040}
\end{eqnarray}
where $q_n=\frac{\|g_n\|_{L^{p+1}(\mathcal{C})}}{\|f\|_{L^{p+1}(\mathcal{C})}}\leq1$.  By \cite[Lemma~2.3]{K2023}, we have
\begin{eqnarray*}
\frac{(q_n^{p+1}+1)^{\frac{2}{p+1}}-1}{q_n^2}\leq2^{\frac{2}{p+1}}-1,
\end{eqnarray*}
which, together with \eqref{eq0009} and \eqref{eqq1040}, implies that
\begin{eqnarray*}
c_{BE}\geq2-2^{\frac{2}{p+1}}.
\end{eqnarray*} 
It contradicts Proposition~\ref{propq0002}.  Thus, we must have $g_n\to0$ strongly in $H^1(\mathcal{C})$.

{\bf Step.~2}\quad We prove that $dist_{H^1(\mathcal{C})}^2(f, \mathcal{Y})>0$ under the assumptions~$(i)$ and $(ii)$.

Again, we suppose the contrary that $dist_{H^1(\mathcal{C})}(f, \mathcal{Y})=0$, then by Step.~1, we have $g_n\to0$ strongly in $H^1(\mathcal{C})$.  It follows that $dist(v_n,\mathcal{Y})\to0$.  Now, we are in the same situation as that in the proof of \cite[Proposition~4.1]{K2023}.  Thanks to Proposition~\ref{propq0001}, we can use the same argument as that used for \cite[Proposition~2]{CFW2013} (see also the proof of \cite[Proposotion~4.1]{WW2022}) to show that $c_{BE}\geq\frac{\lambda_{0,2}-1}{\lambda_{0,2}}$ under the assumptions~$(i)$ and $(ii)$, which contradicts Proposition~\ref{propq0002} under the assumptions~$(i)$ and $(ii)$.  Thus, we must have $dist_{H^1(\mathcal{C})}^2(f, \mathcal{Y})>0$ under the assumptions~$(i)$ and $(ii)$.

{\bf Step.~3}\quad We prove that the variational problem~\eqref{eqq0001} has a minimizer $f$ under the assumptions~$(i)$ and $(ii)$.

Since by Step.~1, $g_n\to0$ strongly in $H^1(\mathcal{C})$ and by Step.~2, $dist_{H^1(\mathcal{C})}^2(f, \mathcal{Y})>0$ under the assumptions~$(i)$ and $(ii)$, the variational problem~\eqref{eqq0001} has a minimizer $f$ under the assumptions~$(i)$ and $(ii)$.
\hfill$\Box$

\section{Acknowledgements}
The research of J. Wei is
partially supported by NSERC of Canada and the research of Y. Wu is supported by NSFC (No. 11971339, 12171470).  The authors also thank Tobias Konig for his helpful discussions to improve this paper.

\section{Appendix: The remaining case}
The remaining case, that is, $N\geq2$ with $a<a_c^*$ and $b_{FS}(a)<b<b_{FS}^*(a)$,  is very special for the variational problem~\eqref{eqq0001}.  On one hand, by Proposition~\ref{propq0002}, the energy estimate $c_{BE}<2-2^{\frac{2}{p+1}}$ still holds for this case.  On the other hand, if we can establish the energy estimate $c_{BE}<\lambda_*=\frac{\lambda_{1,0}-1}{\lambda_{1,0}}$
for this case as that for the cases~$(i)$ and $(ii)$ in Proposition~\ref{propq0002}, then by the same arguments as that used for Theorem~\ref{thmq0001}, we can still prove that $c_{BE}$ is attained in this case, where $\lambda_*$ is given by Proposition~\ref{propq0001}.  In what follows, we shall show that the test function $u=\Psi+\ve\rho_{1,0,l}$ with $\ve\to0$, which seems to be the possiblely optimal test functions according to Proposition~\ref{propq0001}, is invalid in deriving the energy estimate $c_{BE}<\lambda_*$ in this case.

\vskip0.12in

By Proposition~\ref{propq0001}, it is easy to see that $\langle \Psi^{p-2}, \rho_{1,0,l}^3\rangle_{L^2(\mathcal{C})}=0$ for all $l=1,2,\cdots,N$.  Now, as that in the proof of Proposition~\ref{propq0002}, we will have $c_{BE}\leq\lambda_*+o(\ve)$ in the remaining case.
Thus, to go further, we need to expand $\|\Psi+\ve\rho_{0,2,l}\|^2_{L^{p+1}(\mathcal{C})}$ to higher oder terms.  Since by Remark~\ref{rmkq0001}, we can expand $\|\Psi+\ve\rho_{0,2,l}\|^2_{L^{p+1}(\mathcal{C})}$  to arbitrary order terms, by the Taylor expansion, 
\begin{eqnarray}\label{eqq0050}
\|\Psi+\ve\rho_{1,0,l}\|^{p+1}_{L^{p+1}(\mathcal{C})}&=&\|\Psi\|^{p+1}_{L^{p+1}(\mathcal{C})}+\ve(p+1)\langle \Psi^{p}, \rho_{1,0,l}\rangle_{L^2(\mathcal{C})}+\ve^2\frac{(p+1)p}{2}\langle \Psi^{p-1}, \rho_{1,0,l}^2\rangle_{L^2(\mathcal{C})}\notag\\
&&+\ve^3\frac{p(p^2-1)}{6}\langle \Psi^{p-2}, \rho_{1,0,l}^3\rangle_{L^2(\mathcal{C})}
+\ve^4\frac{p(p^2-1)(p-2)}{12}\langle \Psi^{p-3}, \rho_{1,0,l}^4\rangle_{L^2(\mathcal{C})}\notag\\
&&+o(\ve^4).
\end{eqnarray}
It follows from \eqref{eq0006}, $\rho_{1,0,l}\in\mathcal{M}^{\perp}$ and the Taylor expansion once more, that
\begin{eqnarray}
C_{a,b,N}^{-1}\|\Psi+\ve\rho_{1,0,l}\|^2_{L^{p+1}(\mathcal{C})}
&=&C_{a,b,N}^{-1}\|\Psi\|^2_{L^{p+1}(\mathcal{C})}+p\ve^2\langle \Psi^{p-1}, \rho_{1,0,l}^2\rangle_{L^2(\mathcal{C})}\notag\\
&&+\frac{p(p-1)\ve^3}{3}\langle \Psi^{p-2}, \rho_{1,0,l}^3\rangle_{L^2(\mathcal{C})}
+\frac{p(p-1)(p-2)\ve^4}{12}\langle \Psi^{p-3}, \rho_{1,0,l}^4\rangle_{L^2(\mathcal{C})}\notag\\
&&-C_{a,b,N}^{\frac{p+1}{p-1}}\frac{(p-1)p^2\ve^4}{4}(\langle \Psi^{p-1}, \rho_{1,0,l}^2\rangle_{L^2(\mathcal{C})})^2+o(\ve^4).\label{eqq2033}
\end{eqnarray}
Then by Proposition~\ref{propq0001}, \eqref{eqq0050} and \eqref{eqq2033}, we will have the following energy estimate:
\begin{eqnarray}\label{eqq3051}
c_{BE}\leq\lambda_*-\frac{\widehat{\mathbb{Z}}_{a,b,N,l}}{\|\rho_{1,0,l}\|^2_{H^1(\mathcal{C})}}\ve^2+o(\ve^2),
\end{eqnarray}
where we denote
\begin{eqnarray*}
\widehat{\mathbb{Z}}_{a,b,N,l}=\frac{p(p-1)(p-2)}{12}\langle \Psi^{p-3}, \rho_{1,0,l}^4\rangle_{L^2(\mathcal{C})}-C_{a,b,N}^{\frac{p+1}{p-1}}\frac{(p-1)p^2}{4}(\langle \Psi^{p-1}, \rho_{1,0,l}^2\rangle_{L^2(\mathcal{C})})^2.
\end{eqnarray*}
Clearly, if we want to derive the desired energy estimate $c_{BE}<\lambda_*$, we need to show that $\mathbb{Z}_{a,b,N,l}>0$ where $p>2$ is necessary.
Recall that $p=\frac{N+2(1+a-b)}{N-2(1+a-b)}$ with $a\leq b<a+1$,  Thus, we must have $2\leq N\leq 5$ for $p>2$.  It follows that $\widehat{\mathbb{Z}}_{a,b,N,l}<0$ for $N\geq6$, which implies that we can not derive the desired estimate $c_{BE}<\lambda_*$ for $N\geq6$ and $a<a_c^*$ with $b_{FS}(a)<b<b_{FS}^*(a)$ any more by the possiblely optimal test functions $u=\Psi+\ve\rho_{1,0,l}$ as $\ve\to0$.

\vskip0.12in

For $2\leq N\leq5$, we know that $p>2$ is equivalent to 
\begin{eqnarray*}
b<b_{FS}^{**}(a):=a-a_c+\frac{N}{3}.
\end{eqnarray*}
It follows from $a<a_c^*$ with $b_{FS}(a)<b<b_{FS}^*(a)$ that $a_c^{**}<a<a_c^*$ where
\begin{eqnarray*}
a_c^{**}=a_c-\frac{2}{\sqrt{5}}\sqrt{N-1}.
\end{eqnarray*}
Moreover,
$b_{FS}^{**}(a)<b_{FS}^*(a)$ for $a<a_c^{***}$ and $b_{FS}^{**}(a)>b_{FS}^*(a)$ for $a>a_c^{***}$ where 
\begin{eqnarray*}
a_c^{***}=a_c-\frac{\sqrt{3}}{3}\sqrt{N-1}.
\end{eqnarray*}
We remark that since $2\leq N\leq 5$, we have $a_c^{***}<a_c^*$.  Since $p$ is decreasing for $b$, we have
\begin{eqnarray}\label{eqq9099}
q_*(a)\leq p\leq2q_*(a)-1\quad\text{for }a_c^{***}\leq a<a_c^*
\end{eqnarray}
and
\begin{eqnarray}\label{eqq9098}
2\leq p\leq2q_*(a)-1\quad\text{for }a_c^{**}\leq a<a_c^{***},
\end{eqnarray}
where
\begin{eqnarray}\label{eqq9097}
q_*(a)=\sqrt{1+\frac{N-1}{(a_c-a)^2}}.
\end{eqnarray}
We also remark that for $a\in[a_c^{**}, a_c^*)$, 
\begin{eqnarray}\label{eqq9095}
2\leq q_*(a)\leq\bigg(\frac{N+2}{N-2}\bigg)_+\quad\text{for }a_c^{***}\leq a<a_c^*
\end{eqnarray}
and
\begin{eqnarray}\label{eqq9094}
\frac{3}{2}\leq q_*(a)\leq2\quad\text{for }a_c^{**}\leq a<a_c^{***}.
\end{eqnarray}
By Proposition~\ref{propq0001} and  \eqref{eq0026}, we have
\begin{eqnarray*}
\langle \Psi^{p-3}, \rho_{1,0,l}^4\rangle_{L^2(\mathcal{C})}=\bigg(\frac{p+1}{2}(a_c-a)^2\bigg)^{\frac{p-3}{p-1}}\int_{\mathbb{S}^{N-1}}\theta_{1,l}^4d\theta\int_{\bbr}(\cosh(\gamma t))^{-2(\frac{p-3}{p-1}+\frac{2\sqrt{\tau_{a,1}}}{\gamma})}dt
\end{eqnarray*}
and
\begin{eqnarray*}
\langle \Psi^{p-1}, \rho_{1,0,l}^2\rangle_{L^2(\mathcal{C})}=\frac{p+1}{2}(a_c-a)^2\int_{\mathbb{S}^{N-1}}\theta_{1,l}^2d\theta\int_{\bbr}(\cosh(\gamma t))^{-2(1+\frac{\sqrt{\tau_{a,1}}}{\gamma})}dt
\end{eqnarray*}
where $\tau_{a,1}$ is given by \eqref{eqq5699}.  Since by symmetry, we have
\begin{eqnarray*}
\int_{\mathbb{S}^{N-1}}\theta_{1,l}^2d\theta=\frac{1}{N}|\mathbb{S}^{N-1}|\quad\text{and}\quad \int_{\mathbb{S}^{N-1}}\theta_{1,l}^4d\theta=\frac{3}{N(N+2)}|\mathbb{S}^{N-1}|,
\end{eqnarray*}
by \eqref{eqq1041} and the explicit formula of $C_{a,b,N}^{-1}$ given by \cite[Corollary~1,3]{DEL2016}, that is, 
\begin{eqnarray*}
C_{a,b,N}^{-1}=\frac{p+1}{2}(a_c-a)^{\frac{p+3}{p+1}}\bigg(\frac{2\sqrt{\pi}\Gamma\bigg(\frac{p+1}{p-1}\bigg)}{(p-1)\Gamma\bigg(\frac{3p+1}{2(p-1)}\bigg)}\bigg)^{\frac{p-1}{p+1}},
\end{eqnarray*}
we have
\begin{eqnarray}
\widehat{\mathbb{Z}}_{a,b,N,l}&=&(a_c-a)^{\frac{p-5}{p-1}}\frac{p(p-2)|\mathbb{S}^{N-1}|}{2N(N+2)}\bigg(\frac{p+1}{2}\bigg)^{\frac{p-3}{p-1}}\notag\\
&&\times\bigg(\mathbb{B}(\frac{p-3}{p-1}+2\frac{\sqrt{\tau_{a,1}}}{\gamma},\frac12)
-\frac{pD_N\mathbb{B}^2(1+\frac{\sqrt{\tau_{a,1}}}{\gamma},\frac12)}{(p-2)\mathbb{B}(\frac{p+1}{p-1},\frac12)}\bigg),\label{eqq9093}
\end{eqnarray}
where
\begin{eqnarray*}
D_N=\frac{(N+2)|\mathbb{S}^{N-1}|}{N}.
\end{eqnarray*}
Since it is well known that
\begin{eqnarray*}
|\mathbb{S}^{N-1}|=\left\{
\aligned&\frac{2\pi^m}{(m-1)!},\quad N=2m,\\
&\frac{2(2\pi)^m}{(2m-1)!!},\quad N=2m+1,
\endaligned
\right.
\end{eqnarray*}
we have
\begin{eqnarray}\label{eqq9096}
D_N=\left\{
\aligned
&4\pi,\quad N=2,\\
&\frac{20\pi}{3},\quad N=3,\\
&3\pi^2,\quad N=4,\\
&\frac{56\pi^2}{15},\quad N=5.
\endaligned
\right.
\end{eqnarray}
Recall that by the definitions of $\gamma$ and $\tau_{a,1}$ given by Proposition~\ref{propq0001} and \eqref{eqq5699}, respectively, we have 
\begin{eqnarray*}
\frac{\sqrt{\tau_{a,1}}}{\gamma}-\frac{2}{p-1}=\frac{2}{p-1}(q_*(a)-1).
\end{eqnarray*} 
Thus, by \eqref{eqq9099} and \eqref{eqq9098},
\begin{eqnarray*}
1\leq\frac{\sqrt{\tau_{a,1}}}{\gamma}-\frac{2}{p-1}\leq2,
\end{eqnarray*}
where we have also used the monotone property of $q_*(a)$.  As that in the computations for the cases~$(i)$ and $(ii)$ in the proof of Proposition~\ref{propq0002}, by the monotone property of the beta function $\mathbb{B}(m,n)$ in terms of $m$ and the equality $\mathbb{B}(m,n)=\frac{m-1}{m-1+n}\mathbb{B}(m-1,n)$,
\begin{eqnarray}
&&\mathbb{B}(\frac{p-3}{p-1}+2\frac{\sqrt{\tau_{a,1}}}{\gamma},\frac12)
-\frac{pD_N\mathbb{B}^2(1+2\frac{\sqrt{\tau_{a,1}}}{\gamma},\frac12)}{(p-2)\mathbb{B}(\frac{p+1}{p-1},\frac12)}\notag\\
&=&\frac{2\frac{\sqrt{\tau_{a,1}}}{\gamma}-\frac{2}{p-1}}{2\frac{\sqrt{\tau_{a,1}}}{\gamma}-\frac{2}{p-1}+\frac12}\mathbb{B}(2\frac{\sqrt{\tau_{a,1}}}{\gamma}-\frac{2}{p-1},\frac12)-\frac{pD_N\mathbb{B}(1+2\frac{\sqrt{\tau_{a,1}}}{\gamma},\frac12)}{(p-2)\frac{\frac{p+1}{p-1}}{\frac{p+1}{p-1}+\frac12}\times\frac{\frac{p+1}{p-1}+1}{\frac{p+1}{p-1}+\frac32}\mathbb{B}(\frac{p+1}{p-1}+2,\frac12)}\mathbb{B}(1+2\frac{\sqrt{\tau_{a,1}}}{\gamma},\frac12)\notag\\
&\leq&4\bigg(\frac{2q_*(a)-1}{8q_*(a)+p-5}-\frac{2D_Np^2(p+1)}{(p-2)(2p+1)(5p-1)}\bigg)\mathbb{B}(1+2\frac{\sqrt{\tau_{a,1}}}{\gamma},\frac12)\notag\\
&=&\frac{4\mathbb{B}(1+2\frac{\sqrt{\tau_{a,1}}}{\gamma},\frac12)\overline{f}_{a,N}(p)}{(8q_*(a)+p-5)(p-2)(2p+1)(5p-1)},\label{eqq9092}
\end{eqnarray}
where
\begin{eqnarray*}
\overline{f}_{a,N}(p)&=&-2D_Np^4-2(4D_N-5)(2q_*(a)-1)p^3-(17(2q_*(a)-1)+2D_N(8q_*(a)-5))p^2\\
&&-7(2q_*(a)-1)p+2(2q_*(a)-1)
\end{eqnarray*}
with $q_*(a)$ given by \eqref{eqq9097}.  Since $q_*(a)\geq\frac{3}{2}$ and $D_N\geq\pi$ by \eqref{eqq9095}-\eqref{eqq9094} and \eqref{eqq9096}, respectively, we have 
\begin{eqnarray*}
\overline{f}_{a,N}(p)&\leq&-2D_Np^4-2(4D_N-5)(2q_*(a)-1)p^3-(17(2q_*(a)-1)+2D_N(8q_*(a)-5))p^2\\
&\leq&-p^2(2D_Np^2+2(4D_N-5)(2q_*(a)-1)p)\\
&<&0
\end{eqnarray*}
for $p>1$.
It follows from \eqref{eqq9093} and \eqref{eqq9092} that $\widehat{\mathbb{Z}}_{a,b,N,l}<0$ for $2\leq N\leq5$ and $a<a_c^*$ with $b_{FS}(a)<b<b_{FS}^*(a)$, which implies that we also can not derive the desired estimate $c_{BE}<\lambda_*$ from \eqref{eqq3051} for $2\leq N\leq5$ and $a<a_c^*$ with $b_{FS}(a)<b<b_{FS}^*(a)$.


\begin{thebibliography}{100}
\bibitem{A1976}
T. Aubin, Probl\`emes isop\'erim\'etriques de Sobolev.  {\it J. Differential Geometry,} {\bf11} (1976), 573--598.

\bibitem{BE1991}
G. Bianchi, H. Egnell, A note on the Sobolev inequality.  {\it J. Funct. Anal.,} {\bf100} (1991), 18--24.

\bibitem{BL1985}
H. Brezis, E. Lieb, Sobolev inequalities with remainder terms. {\it J. Funct. Anal.,} {\bf62} (1985), 73--86.

\bibitem{CKN1984}
L. Caffarelli, R. Kohn, L. Nirenberg, First order interpolation inequalities with weights.  {\it Compos. Math.,} {\bf53} (1984), 259--275.

\bibitem{CW2001}
F. Catrina, Z.-Q. Wang, On the Caffarelli-Kohn-Nirenberg inequalities: sharp constants, existence (and nonexistence), and symmetry of extremal functions. {\it Comm. Pure Appl. Math.,} {\bf54} (2001), 229--258.

\bibitem{CFW2013}
S. Chen, R. L. Frank, T. Weth, Remainder terms in the fractional Sobolev inequality, {\it Indiana Univ. Math. J.,} {\bf62} (2013), 1381--1397.

\bibitem{CC1993}
K. Chou, W. Chu, On the best constant for a weighted Sobolev-Hardy inequality.  {\it J. London Math. Soc.,} {\bf48} (1993), 137--151.

\bibitem{DT2023}
S. Deng, X. Tian, On the stability of Caffarelli-Kohn-Nirenberg inequality in $\bbr^2$, 
preprint, arXiv2308.04111v1 [Math. AP].

\bibitem{DEFFL2022}
J. Dolbeault, M. J. Esteban, A. Figalli, R. L. Frank, M. Loss, Stability for the Sobolev inequality with explicit constants. Preprint, arXiv:2209.08651 [Math. AP].

\bibitem{DELT2009}
J. Dolbeault, M. J. Esteban, M. Loss, and G. Tarantello, On the symmetry of extremals for the Caffarelli-Kohn-Nirenberg inequalities.  {\it Adv. Nonlinear Stud.,} {\bf9} (2009), 713--726.

\bibitem{DEL2012}
J. Dolbeault, M. J. Esteban, M. Loss, Symmetry of extremals of functional inequalities via spectral estimates for linear operators.  {\it J. Math. Phys.,} {\bf53} (2012), article 095204, 18 pp.

\bibitem{DEL2016}
J. Dolbeault, M. J. Esteban, M. Loss, Rigidity versus symmetry breaking via nonlinear flows on cylinders and Euclidean spaces, {\it Invent. math.,} {\bf206} (2016), 397--440.

\bibitem{DET2008}
J. Dolbeault, M. J. Esteban, G. Tarantello, The role of Onofri type inequalities in the symmetry properties of extremals for Caffarelli-Kohn-Nirenberg inequalities, in two space dimensions, {\it Ann. Sc. Norm. Super. Pisa Cl. Sci.,} {\bf5} (2008), 313-341.

\bibitem{DE2022}
J. Dolbeault, M. J. Esteban, Hardy-Littlewood-Sobolev and related inequalities: stability, EMS Press, Berlin, 2022, 247–268.

\bibitem{F2013}
A. Figalli, Stability in geometric and functional inequalities, European Congress of Mathematics, 585--599, Eur. Math. Soc., Zurich, 2013.

\bibitem{FG2021}
A. Figalli, F. Glaudo, On the Sharp Stability of Critical Points of the Sobolev Inequality.  {\it Arch. Rational Mech. Anal.,} {\bf237} (2020), 201--258.

\bibitem{FS2003}
V. Felli, M. Schneider, Perturbation results of critical elliptic equations of Caffarelli-Kohn-Nirenberg type. {\it J. Differential Equations,} {\bf191} (2003), 121--142.

\bibitem{FLW2022}
R. L. Frank, A. Laptev, T. Weidl, Schr\"odinger Operators: Eigenvalues and Lieb-Thirring Inequalities,
In: Cambridge Studies in Advanced Mathematics 200. Cambridge University Press (2022).

\bibitem{FP2023}
R. L. Frank, J. W. Peteranderl, Degenerate stability of the Caffarelli-Kohn-Nirenberg inequality along the Felli-Schneider curve, preprint, arXiv:2308.07917 [math.AP]


\bibitem{F2022}
R. L. Frank, Degenerate stability of some Sobolev inequalities, {\it Ann. Inst. H. Poincare Anal. Non Lineaire,} {\bf 39} (2022), 1459-1484.

\bibitem{F2023}
R. L. Frank, The sharp Sobolev inequality and its stability: an introduction, preprint, arXiv2304.03115 [math.AP].

\bibitem{GLKO1993}
A. Gonzalez-Lopez, N. Kamran, P.J. Olver, Normalizability of one-dimensional quasi-exactly solvable Schro\"dinger operators, {\it Comm. Math. Phys.,} {\bf153} (1993), 117--146.

\bibitem{K2022}
T. Konig, On the sharp constant in the Bianchi-Egnell stability inequality. Preprint,
arXiv:2210.08482

\bibitem{K2023}
T. Konig, Stability for the Sobolev inequality: Existence of a minimizer, preprint, arXiv2211.14185v3 [Math. AP].

\bibitem{K2023-1}
T. Konig, An exceptional property of the one-dimensional Bianchi-Egnell inequality, preprint, arXiv:2308.16794v1 [math.AP]

\bibitem{LL1958}
L.D. Landau, E.M. Lifshitz, Quantum Mechanics: Non-Relativistic Theory, Theoretical Physics, Vol. 3, Pergamon Press Ltd., London, Paris, 1958.

\bibitem{L1983}
E. Lieb, Sharp constants in the Hardy-Littlewood-Sobolev and related inequalities. {\it Ann. of Math. (2),} {\bf118} (1983), 349--374.

\bibitem{LW2004}
C.-S. Lin, Z.-Q. Wang, Symmetry of extremal functions for the Caffarrelli-Kohn-Nirenberg inequalities. {\it Proc. Amer. Math. Soc.,} {\bf132} (2004), 1685--1691.

\bibitem{NU1988}
A.F. Nikiforov, V.B. Uvarov, Special Functions of Mathematical Physics, Birkhäuser Verlag, Basel, 1988, xviii+427 pp.

\bibitem{S2007}
M. Schneider, A priori estimates for the scalar curvature equation on $\mathbb{S}^3$, 
{\it Calc. Var.,} 29 (2007), 521--560.

\bibitem{T1976}
G. Talenti, Best constant in Sobolev inequality.  {\it Ann. Mat. Pura Appl. (4),} {\bf110} (1976), 353--372.

\bibitem{WW2022}
J. Wei, Y. Wu, On the stability of the Caffarelli–Kohn–Nirenberg inequality, {\it Math. Ann.,} {\bf384} (2022), 1509--1546.
\end{thebibliography}
\end{document}